\documentclass[12pt,letterpaper,reqno]{amsart}

\usepackage{times}
\usepackage[T1]{fontenc}
\usepackage{mathrsfs}
\usepackage{latexsym}
\usepackage[dvips]{graphics}
\usepackage{epsfig}
\usepackage{amsmath,amsfonts,amsthm,amssymb,amscd, hyperref}
\usepackage{verbatim}
\input amssym.def
\input amssym.tex

\newcommand{\supp}{{\rm supp}}

\newcommand\be{\begin{equation}}
\newcommand\ee{\end{equation}}
\newcommand\bea{\begin{eqnarray}}
\newcommand\eea{\end{eqnarray}}
\newcommand\bi{\begin{itemize}}
\newcommand\ei{\end{itemize}}
\newcommand\ben{\begin{enumerate}}
\newcommand\een{\end{enumerate}}
\newcommand\bc{\begin{center}}
\newcommand\ec{\end{center}}
\newcommand\ba{\begin{array}}
\newcommand\ea{\end{array}}

\def\notdiv{\ \mathbin{\mkern-8mu|\!\!\!\smallsetminus}}

\newcommand{\R}{\ensuremath{\mathbb{R}}}

\newcommand{\Z}{\ensuremath{\mathbb{Z}}}

\newtheorem{thm}{Theorem}[section]

\newtheorem{lem}[thm]{Lemma}
\newtheorem{prop}[thm]{Proposition}

\newtheorem{rek}[thm]{Remark}

\newcommand{\twocase}[5]{#1 \begin{cases} #2 & \text{{\rm #3}}\\ #4
&\text{{\rm #5}} \end{cases}   }

\newcommand{\mattwo}[4]
{\left(\begin{array}{cc}
                        #1  & #2   \\
                        #3 &  #4
                          \end{array}\right) }

\newcommand{\intii}{\int_{-\infty}^\infty}

\numberwithin{equation}{section}

\begin{document}

\title[Low-lying Zeros of Cuspidal Maass Forms]{Low-lying Zeros of Cuspidal Maass Forms}

\author{Nadine Amersi}\email{n.amersi@ucl.ac.uk}
\address{Department of Mathematics, University College London, London, WC1E 6BT}

\author{Geoffrey Iyer}\email{geoff.iyer@gmail.com}
\address{Department of Mathematics, University of Michigan, Ann Arbor, MI 48109}

\author{Oleg Lazarev}\email{olazarev@Princeton.edu}
\address{Department of Mathematics, Princeton University, Princeton, NJ 08544}

\author{Steven J. Miller}\email{sjm1@williams.edu, Steven.Miller.MC.96@aya.yale.edu}
\address{Department of Mathematics and Statistics, Williams College, Williamstown, MA 01267}

\author{Liyang Zhang}\email{lz1@williams.edu}
\address{Department of Mathematics and Statistics, Williams College, Williamstown, MA 01267}

\subjclass[2000]{11M26 (primary), 11M41, 15A52 (secondary).}

\keywords{$1$-level density, Maass forms, low-lying zeros, Kuznetsov trace formula}

\date{\today}

\thanks{We thank Gergely Harcos, Andrew Knightly, Stephen D. Miller and Peter Sarnak for helpful conversations on an earlier version.  This work was done at the 2011 SMALL Undergraduate Research Project at Williams College, funded by NSF GRANT DMS0850577 and Williams College; it is a pleasure to thank them for their support. The first named author was also supported by the Mathematics Department of University College London, and the fourth named author was partially supported by NSF grant DMS0970067.}

\begin{abstract}  The Katz-Sarnak Density Conjecture states that the behavior of zeros of a family of $L$-functions near the central point (as the conductors tend to zero) agree with the behavior of eigenvalues near 1 of a classical compact group (as the matrix size tends to infinity). Using the Petersson formula, Iwaniec, Luo and Sarnak \cite{ILS} proved that the behavior of zeros near the central point of holomorphic cusp forms agree with the behavior of eigenvalues of orthogonal matrices for suitably restricted test functions. We prove a similar result for level 1 cuspidal Maass forms, the other natural family of ${\rm GL}_2$ $L$-functions. We use the explicit formula to relate sums of our test function at scaled zeros to sums of the Fourier transform at the primes weighted by the $L$-function coefficients, and then use the Kuznetsov trace formula to average the Fourier coefficients over the family. There are numerous technical obstructions in handling the terms in the trace formula, which are surmounted through the use of smooth weight functions for the Maass eigenvalues and results on Kloosterman sums and Bessel and hyperbolic functions.

\end{abstract}


\maketitle
\setcounter{equation}{0}







\tableofcontents


\section{Introduction}

The distribution of zeros of $L$-functions play an important role in numerous problems in number theory, from the distribution of primes \cite{Con,Da} to the size of the class group \cite{CI,Go,GZ}. In the 1970s, Montgomery \cite{Mon} observed that the pair correlation of the zeros of $\zeta(s)$, for suitable test functions, agree with that of the eigenvalues of the Gaussian Unitary Ensemble (GUE). This suggested a powerful connection between number theory and random matrix theory (see \cite{FM,Ha} for some of the history), which was further supported by Odlyzko's investigations \cite{Od1,Od2} showing agreement between the spacings of zeros of $\zeta(s)$ and the eigenvalues of the GUE. Later studies by Katz and Sarnak \cite{KaSa1,KaSa2} showed more care is needed. Specifically, although the $n$-level correlations agree for suitable test functions \cite{Hej,RS} and are the same for all classical compact groups, the behavior of the eigenvalues near 1 is different for unitary, symplectic and orthogonal matrices. This led to the Katz-Sarnak Density Conjecture, which states that the behavior of zeros near the central point in a family of $L$-functions (as the conductors tend to infinity) agree with the behavior of eigenvalues near 1 of a classical compact group (as the matrix size tends to infinity). For suitable test functions, this has been verified in many families, including Dirichlet characters, elliptic curves, cuspidal newforms, symmetric powers of ${\rm GL}(2)$ $L$-functions, and certain families of ${\rm GL}(4)$ and ${\rm GL}(6)$ $L$-functions; see for example \cite{DM1,DM2,FI,Gao,Gu,HM,HR,ILS,KaSa2,Mil1,MilPe,OS,RR,Ro,Rub,Ya,Yo2}. This correspondence between zeros and eigenvalues allows us, at least conjecturally, to assign a definite symmetry type to each family of $L$-functions (see \cite{DM2} for more on identifying the symmetry type of a family).

For this work, the most important families studied to date are holomorphic cusp forms. Using the Petersson formula (and a delicate analysis of the resulting Bessel-Kloosterman term), Iwaniec, Luo and Sarnak \cite{ILS} proved that the limiting behavior of the zeros near the central point of holomorphic cusp forms agree with the eigenvalues of orthogonal matrices for suitably restricted test functions. In this paper we look at the other ${\rm GL}_2$ family of $L$-functions, namely Maass forms. Specifically, we study the family of level 1 cuspidal Maass Forms. We quickly recall their properties; see \cite{Iw2,IK,Liu,LY2} for details. A Maass form on the group ${\rm SL}_2(\Z)$ is a smooth function $u \neq 0$ on the upper half-plane $\mathbb{H}$  satisfying:

\begin{enumerate}
\item For all $g \in {\rm SL}_2(\Z)$ and all $z \in \mathbb{H}$, $u(gz)=u(z)$, with ${\rm SL}_2(\Z)$ acting on $\mathbb{H}$ by $gz= \frac{az+b}{cz+d}$ for $g = \tiny\mattwo{a}{b}{c}{d}\normalsize$;
\item $u$ is an eigenfunction of the non-Euclidean Laplacian $\Delta = -y^2 \left(\frac{\partial^2}{\partial x^2} + \frac{\partial^2}{\partial y^2} \right)$, with $\Delta u = \lambda u$, where the eigenvalue $\lambda = s(1-s)$;
\item there exists a positive integer $k$ such that $u(z)\ll y^k$ as $y \rightarrow +\infty$.
\end{enumerate}

If $u$ is a cuspidal Maass form, then for all $z \in \mathbb{H}$ we have \be \int_0^1 u\left(\mattwo{1}{b}{0}{1}z\right)db \ = \ 0. \ee The Selberg eigenvalue conjecture states that the eigenvalues of the Laplacian on a congruence group are at least $1/4$. Though open in general, it has been proved for ${\rm SL}_2(\Z)$ (see for instance \cite{DI}). This allows us to write $\lambda = s(1-s)$ as $\lambda_j = (\frac12+it_j)(\frac12-it_j)$ with $t_j \in \R$. By Weyl's Law, the number of $\lambda_j = (\frac12 + it_j)(\frac12-it_j)$ at most $x$ is $x/12 + O(x^{1/2}\log x)$. In particular, this means that the average spacing between eigenvalues $t_j$ around $T$ is on the order of $1/T$.

If $u_j$ is a cuspidal Maass form on ${\rm SL}_2(\Z)$, it has a Fourier expansion
\be
u_j(z)\ =\ \cosh(t_j) \sum_{n\ne 0} \sqrt{y}\lambda_j(n) K_{it_j}(2\pi |n|y)e^{2\pi i n x},
 \ee
and its norm \be ||u_j||^2 \ = \ \int_{{\rm SL}_2(\Z)\setminus\mathbb{H}} |u_j|^2 \frac{dxdy}{y^2} \ee satisfies (see \cite{HL,Iw1}) \be\label{eq:boundsnormuj} |\lambda_j|^{-\epsilon} \ \ll \ ||u_j||^2 \ \ll \ |\lambda_j|^{\epsilon} \ee for any $\epsilon>0$. There are many ways to normalize $u_j$. We choose to normalize $u_j$ by having $\lambda_j(1)=1$; while this is useful as it leads to the Fourier coefficients being multiplicative, it causes some problems with the normalizations needed to apply the Kuznetsov trace formula to average over these coefficients, and influences our choice of weight function below.

We define the $L$-function attached to $u_j$ by
\begin{equation}
L(s,u_j)\ =\ \sum_{n\ge 1} \frac{\lambda_j(n)}{n^s},
\end{equation}
with Euler product
 \begin{eqnarray}
L(s,u_j)
&\ = \ & \prod_p\left({1-\frac{\lambda_j(p)}{p^s}+\frac{1}{p^{2s}}}\right)^{-1}\nonumber\\
&=& \prod_p\left(1-\frac{\alpha_j(p)}{p^s}\right)^{-1}\left(1-\frac{\beta_j(p)}{p^s}\right)^{-1}, \label{eqn: product}
\end{eqnarray}
where the Satake parameters $\alpha_j(p), \beta_j(p)$ satisfy
\begin{equation} \label{eqn: alphabeta}
\alpha_j(p)+\beta_j(p) \ =  \ \lambda_j(p), \ \ \ \ \alpha_j(p)\beta_j(p) \ = \ 1.
\end{equation}

We study the low-lying zeros of the L-functions associated to Maass forms. Our main statistic for studying these zeros is the 1-level density, which we now describe. Let $\phi$ be an even Schwartz function such that the Fourier transform $\widehat{\phi}$ of $\phi$ has compact support; that is,
\be
\widehat{\phi}(y)\ =\ \int_{-\infty}^\infty \phi(x)e^{-2\pi i xy}dx
\ee
and there is a $\sigma<\infty$ such that $\widehat{\phi}(y) = 0$ for $y$ outside $(-\sigma, \sigma)$.

The $1$-level density of $L(s, u_j)$ is
\begin{equation} \label{eqn: 1leveldef}
D_1(u_j; \phi) \ = \
\sum_{\gamma_j}
 \phi \left(\frac{\log R}{2\pi}\gamma_{j} \right),
\end{equation}
where $\rho_j = \frac{1}{2}+i\gamma_{j}$ are the zeros of $L(s, u_j)$, and $\log R$ is a rescaling parameter related to the average log-conductor in the weighted family (to be defined carefully below). As $\phi$ is a Schwartz function, most of the contribution comes from the zeros near the central point, i.e., the low-lying zeros. To complete the determination of the symmetry type, we'll also need to study the $2$-level density, which is defined similarly in \S\ref{sec:2level}.

Similar to how the harmonic weights facilitate applications of the Petersson formula to average the Fourier coefficients of cuspidal newforms (see for instance \cite{ILS,MilMo}), we introduce a nice, even weight function $h_T$ to smooth the sum over the Maass forms. The two most interesting cases to investigate is an expanding window centered at the origin, or two windows at $\pm T$ with widths significantly smaller than $T$ (the narrower the better). \\

\textbf{For definiteness, in this paper $h_T$ will always refer either to the weight function}
\begin{equation}\label{eq:defht1}
h_{1,T}(r) \ = \  \exp\left(-t^2/T^2\right)
\end{equation} \textbf{(which is proportional to a Gaussian centered at zero with standard deviation $T/\sqrt{2}$) or}
\begin{equation}\label{eq:defht2}
h_{2,T}(r) \ = \  \frac12 h\left(\frac{r-T}{L}\right) +  \frac12 h\left(\frac{r+T}{L}\right),
\end{equation}
\textbf{where $h$ is a non-negative even Schwartz function, integrates to 1, has $\supp(\widehat{h}) \subset (-1,1)$, and for some $\eta \in (0,1)$ we have $\frac{\pi}{2\log T} < (1-\eta)L$ and $L = o(T^{1/8})$. We denote the size constraints on $L$ by $\frac{\pi}{2\log T} < (1-\eta)L = o(T^{1/8})$.} \\

\begin{rek} While there is no need to normalize $h_T$, as we divide by the sum of the weights, we chose the two normalizations above to simplify the main terms in the hyperbolic tangent integral in the Kuznetsov formula (see Lemma \ref{lem:findingRprep}). In order to use the Kuznetsov trace formula, there must be a $\delta>0$ such that our weight function satisfies: (1) $h_T(t)=h_T(-t)$, (2) $h_T$ is holomorphic in $|{\rm Im}(t)| \le 1/2 + \delta$, and (3) $h(t) \ll (|t|+1)^{-2-\delta}$; both of our test functions satisfy these conditions. Note the second test function has stronger conditions. This is due to the fact that we are studying Maass forms concentrated around a growing height with a varying width, and the support restriction on $\widehat{h}$ forces the Kloosterman sum in the Kuznetsov formula to have finitely many terms and hence converge. Our analysis can be readily modified to study more general weight functions concentrated about 0. \end{rek}

\begin{rek} The upper bound for $L$ is harmless, as we want to take the smallest possible value of $L$ as we are trying to localize to eigenvalues near $\pm T$; thus what matters is the lower bound on $L$, not the upper (which with slightly more work we could increase). The effect of $L$ and $T$ is that only the eigenvalues within essentially $L$ units of $T$ contribute to the 1-level density. By Weyl's Law, the average spacing between eigenvalues around $T$ is on the order of $1/T$. Thus even for the smallest $L$ we take we are covering many times the average spacing, and thus are studying a large number of eigenvalues. This scale is similar to what has been found in other problems (see for instance \cite{Sar}). \end{rek}

We consider the weighted $1$-level density of the family of cuspidal Maass forms on ${\rm SL}_2(\Z)$:
\be
\lim_{T\to\infty} \mathcal{D}_1(\phi, h_T) \ = \ \lim_{T\to\infty} \frac{1}{\sum_{j} \frac{h_T(t_j)}{\|u_j\|^2}} \sum_{j} \frac{h_T(t_j)}{\|u_j\|^2}
D_1(u_j; \phi). \ee
That is, we care about averaging $D_1(u_j; \phi)$ over `families' of increasing eigenvalues.

We use this weighting to facilitate applying the Kuznetsov Trace Formula (see for example Theorem 16.3 of \cite{IK}), which says that for such normalized $u_j$, we have
\begin{eqnarray}\label{eq:kuznetsovtraceformula}
&&\sum_{j} \frac{h(t_j)}{\|u_j\|^2} \lambda_j(m)\overline{\lambda_j(n)}  +\frac{1}{4\pi} \int_{-\infty}^\infty \overline{\tau(m,r)} \tau(n,r) \frac{h(r)}{\cosh( \pi r)}  dr\nonumber\\
&& =\ \frac{\delta_{n,m}}{\pi^2} \int_{-\infty}^\infty r \tanh(r)h(r)dr + \frac{2i}{\pi} \sum_{c\ge 1} \frac{S(m,n;c)}{c}
\int_{-\infty}^\infty J_{2ir}\left( \frac{4\pi \sqrt{mn}}{c} \right) \frac{h(r)r}{\cosh(\pi r)} dr, \nonumber\\
\end{eqnarray}
where
\begin{eqnarray}
\tau(m,r ) & = & \pi^{\frac{1}{2}+ir}\Gamma(1/2+ir)^{-1}\zeta(1+2ir)^{-1}m^{-1/2}\sum_{ ab =|m| } \left( \frac{a}{b}\right)^{ir} \nonumber\\
S(m,n;c) & = & \displaystyle \sum_{x = 0 \atop  \gcd(x,c)=1}^{c-1}e^{2\pi i(nx+mx^\ast)/c} \nonumber\\
J_{2ir}(x)& = &\sum_{m=0}^\infty\frac{(-1)^m}{m!\Gamma(m+ir+1)}\left(\frac{x}{2}\right)^{2m+ir}
\end{eqnarray} (with $x^\ast$ the multiplicative inverse of $x$ modulo $c$). As we have chosen to normalize our $L$-functions so that $\lambda_j(1) = 1$, instead of normalizing so that the $L^2$-norm is one, our conventions differ from some authors. It is harmless to pass from one to the other, though, as all we need to do is replace $\lambda_j(k)$ with $\lambda_j(k)/||u_j||$; as we always have a product of two Fourier coefficients, we have $||u_j||^2$ in the denominator. See Chapter 15 of \cite{IK} for a proof.\footnote{There are many different normalizations (and frequently minor typos), and thus some care is needed in comparing formulas from different works. For example, there is a typo in the definition of $\phi(n,s)$ (it should be $\zeta(2s)^{-1}$) in Chapter 14 of \cite{IK}, which is then used in the Kuznetsov formula in Chapter 15 of \cite{IK}; the typo would lead to evaluating the reciprocal of the Riemann zeta function on the critical line and not the edge of the critical strip. While this factor is correctly stated in \cite{Liu}, there the author drops the factor of $r$ in the integral with $\tanh(r)$, and absent this factor the integral is trivially zero as the integrand is now odd.}

Our main result is the following.
\begin{thm} \label{thm:levelone}
Let $\phi$ be an even Schwartz function such that the Fourier transform $\widehat{\phi}$ of $\phi$ has compact support in $(-\sigma, \sigma)$, let $h_T$ be either $h_{1,T}$ (equation \eqref{eq:defht1}) or $h_{2,T}$ (equation \eqref{eq:defht2}), and let $\frac{\pi}{2\log T} < (1-\eta)L = o(T^{1/8})$. Then the weighted 1-level density of the family of level 1 Maass forms is
\begin{eqnarray}
& & \frac{1}{\sum_j \frac{h_T(t_j)}{\|u_j\|^2}} \sum_j \frac{h_T(t_j)}{\|u_j\|^2} D_1(u_j; \phi) \nonumber\\ & & \twocase{= \ \frac{\phi(0)}{2} + \widehat{\phi}(0)+ O\left(\frac{\log\log R}{\log R}\right) +}{O(T^{\sigma(3/2 + \epsilon) -1/4})}{if $h = h_{1,T}$}{O(T^{\sigma(3/2+\epsilon)-\eta})}{if $h=h_{2,T}$;} \ \ \
\end{eqnarray} this agrees with the scaling limit of orthogonal matrices for $\sigma < 1/6$ if $h=h_{1,T}$ and $\sigma < \frac23\eta$ for $h=h_{2,T}$. As we may take $\eta$ arbitrarily close to 1, we may take the support in the second case to be arbitrarily smaller than $2/3$. \end{thm}

\begin{rek} We are able to obtain significantly better support for the second case as the eigenvalues are concentrated in small bands about $\pm T$. This assists us in bounding the contribution from the Bessel-Kloosterman terms from the Kuznetsov formula, as this forces the arguments we evaluate at to be quite large. The situation is markedly different in the first case. 
\end{rek}

Unfortunately the 1-level densities of the three orthogonal groups are indistinguishable from each other if the support of $\widehat{\phi}$ is less than 1, though they are distinguishable from the unitary and symplectic groups. While we expect the symmetry type of our family to be governed by the distribution of the signs of the functional equation, Theorem \ref{thm:levelone} is insufficient to prove this and determine which orthogonal group governs the symmetry.

Our support restriction for $\widehat{\phi}$ are due to bounds on non-contributing terms from the Kuznetsov Trace Formula, which is significantly less tractable than the Petersson formula (which is used to average over the Fourier coefficients for holomorphic cuspidal forms). In particular, the small support comes from the Bessel-Kloosterman piece of the Kuznetsov formula.


To surmount these difficulties and determine which orthogonal group is associated to Maass form L-functions,
we turn to the $2$-level density. Miller \cite{Mil0,Mil1} showed that while the 1-level densities for the three orthogonal groups are the same if the support is less than 1, their 2-level densities are different for arbitrarily small support, with the difference related to the percentage of elements with odd functional equation. We therefore study the average $2$-level density, which is given by
\begin{eqnarray}
\mathcal{D}_2(\phi_1,\phi_2,h_T) &\ := \ & \frac{1}{\sum_{u_j \in \mathcal{F}} \frac{h_T(t_j)}{\| u_j \|^2}}
\sum_{u_j \in \mathcal{F}} \frac{h_T(t_j)}{\| u_j \|^2}
\sum_{j_1 \ne \pm j_2}
\phi_1\left(\frac{\log R}{2\pi} \gamma_{j_1}\right)
\overline{\phi_2\left(\frac{\log R}{2\pi} \gamma_{j_2}\right)}.\nonumber\\
\end{eqnarray}

We find
\begin{thm}\label{thm: leveltwo}
Let $h_T$ be as in Theorem \ref{thm:levelone}, $\frac{\pi}{2\log T} < (1-\eta)L = o(T^{1/8})$, $\phi_1, \phi_2$ be even Schwartz functions such that their Fourier transforms have compact support in $(-\sigma, \sigma)$, and $N(-1, \mathcal{F})$ denote the weighted percentage of Maass forms with odd sign in the functional equation, \be
\mathcal{N}(-1)\ =\ \frac{1}{\sum_{j} \frac{h_T(t_j)}{\| u_j \|^2}}
 \sum_{j; \epsilon_j = -1} \frac{h_T(t_j)}{\| u_j \|^2}.
\ee
Then with $\sigma<\frac{1}{3} \eta$ (for $h_{2.T}$) or with $\sigma<1/12$ (for $h_{1,T}$), the 2-level density is
\begin{eqnarray}
\lim_{T\rightarrow \infty}\mathcal{D}_2(\phi_1,\phi_2,h_T)&\ =\ & \prod_{i=1}^2  \left[
\frac{\phi_i(0)}{2}+
\widehat{\phi_i}(0)\right] + \frac{1}{2}\int_{-\infty }^\infty  |z|\widehat{\phi}_1(z)\widehat{\phi}_2(z)dz\nonumber\\
 & & \ \ -\ 2 \left(\frac{\phi_1(0)\phi_1(0)}{2}+ \widehat{\phi_1\phi_2} (0) - (\phi_1 \phi_2)(0) N(-1, \mathcal{F})\right),\nonumber\\
\end{eqnarray} and agrees only with the family of orthogonal matrices where the weighted distribution of sign tends to $\mathcal{N}(-1)$. As we may take $\eta$ arbitrary close to $1$, we may take the support to be arbitrarily smaller than $1/3$ for weight function $h_{2,T}$.
\end{thm}

The paper is organized as follows. In \S\ref{sec:explicit} we discuss the explicit formula and derive the expansion for the 1-level density, Lemma \ref{lem:1levelexpansion}. We review the Kuznetsov Trace Formula in \S\ref{sec: Kuznetsov}, analyzing many of the terms. We then calculate the 1-level density in \S\ref{sec:1level} and the 2-level density in \S\ref{sec:2level}.



\section{Expansion for $1$-Level Density}\label{sec:explicit}

The main result of this section is an expansion for the weighted 1-level density for the family of Maass forms, which is given in Lemma \ref{lem:1levelexpansion}. We first determine the 1-level density for the $L$-function of a single Maass form. Recall that the 1-level density (see \eqref{eqn: 1leveldef}) is a sum over the zeros of the $L$-function of the form $\sum_{\rho} H(\rho)$ for some function $H$; later, we let $H(s)=\phi(\frac{y\log R}{2\pi})$ for $\phi$ an even Schwartz function with Fourier transform compactly supported. We do not need to assume the Generalized Riemann Hypothesis in the arguments below, though its veracity (the non-trivial zeros are of the form $1/2+i\gamma$) leads to a nice spectral interpretation of the 1-level density. Here $R$ is a global scaling parameter, which we determine later in Lemma \ref{lem:findingR}.

The $L$-functions satisfy a function equation:
\begin{eqnarray}
\Lambda(s,u_j) \ = \ \pi^{-s}\Gamma\left(\frac{s+\epsilon_j +it_j}{2}\right)\Gamma\left(\frac{s+\epsilon_j-it_j}{2}\right)L(s,u_j)
\ = \ (-1)^{\epsilon_j}\Lambda(1-s,u_j). \nonumber\\
\end{eqnarray}
Using this functional equation and contour integration, we have the following `explicit formula':
\begin{eqnarray}\label{eqn: explicit}
\sum_{\rho}H(\rho)&= & \frac{1}{2\pi i}\int_{(\frac{3}{2})}\left(-\log(\pi)+\frac{\Gamma'\left(\frac{s+\epsilon_j+it_j}{2}\right)}{2\Gamma \left(\frac{s+\epsilon_j+it_j}{2}\right)} +\frac{\Gamma'\left(\frac{s+\epsilon_j-it_j}{2}\right)}{2\Gamma\left(\frac{s+\epsilon_j-it_j}{2}\right)}\right) [H(s)+H(1-s)]ds\nonumber\\&& \ \ -\
\frac{1}{2\pi i}\int_{(\frac{3}{2})}\left(\sum_p\log p\sum_{k=1}^\infty\frac{\alpha_j^k(p)+\beta_j^k(p)}{p^{sk}}\right)[H(s)+H(1-s)]ds.
\end{eqnarray}
The proof is standard, and proceeds similarly to that for other families; see for example \cite{ILS,RS}. The analysis is thus reduced to understanding the first integral, involving the Gamma factors, and the second integral, involving the Satake parameters $\alpha_j(p), \beta_j(p)$.

\subsection{Analysis of the Gamma factors}

\begin{lem} The Gamma term in \eqref{eqn: explicit} contributes
\begin{eqnarray}
&&
\frac{1}{2\pi i}\int_{(\frac{3}{2})}\left(-\log(\pi)+\frac{\Gamma'\left(\frac{s+\epsilon_j+it_j}{2}\right)}{2\Gamma\left(\frac{s+\epsilon_j+it_j}{2}\right)}
+\frac{\Gamma'\left(\frac{s+\epsilon_j-it_j}{2}\right)}{2\Gamma\left(\frac{s+\epsilon_j-it_j}{2}\right)}\right)[H(s)+H(1-s)]ds\nonumber\\
&& = \ \widehat{\phi}(0) \frac{\log(1+t_j^2)}{\log R} + O\left(\frac{1}{\log R}\right).
\end{eqnarray}
\end{lem}

\begin{proof} We shift the integral of the Gamma term in \eqref{eqn: explicit} to critical line. For $s=\frac{1}{2}+iy$, $H(1-s)=H(s)=\phi(\frac{y\log R}{2\pi})$. Since $\Gamma$ has no zeroes or poles if $\mbox{Re}(z)>0$, there is no residue and the integrals are equal. We find
\begin{eqnarray}
&&\frac{1}{2\pi i}\int_{(\frac{3}{2})}\left(-\log(\pi)+\frac{\Gamma'\left(\frac{s+\epsilon_j+it_j}{2}\right)}{2\Gamma\left(\frac{s+\epsilon_j+it_j}{2}\right)}+\frac{\Gamma'\left(\frac{s+\epsilon_j-it_j}{2}\right)}{2\Gamma\left(\frac{s+\epsilon_j-it_j}{2}\right)}\right)[H(s)+H(1-s)]ds\nonumber\\
&&=\frac{1}{2\pi i}\int_{(\frac{1}{2})}\left(-\log(\pi)+\frac{\Gamma'\left(\frac{s+\epsilon_j+it_j}{2}\right)}{2\Gamma\left(\frac{s+\epsilon_j+it_j}{2}\right)}+\frac{\Gamma'\left(\frac{s+\epsilon_j-it_j}{2}\right)}{2\Gamma\left(\frac{s+\epsilon_j-it_j}{2}\right)}\right)2H(s)ds\nonumber\\
&&=\frac{1}{\pi}\int_{-\infty}^\infty\left(-\log(\pi)+\frac{\Gamma'(\frac{1}{4}+\frac{\epsilon_j}{2}+\frac{i}{2}(y+t_j))}{2\Gamma(\frac{1}{4}+\frac{\epsilon_j}{2}+\frac{i}{2}(y+t_j))}+\frac{\Gamma'(\frac{1}{4}+\frac{\epsilon_j}{2}+\frac{i}{2}(y-t_j))}{2\Gamma(\frac{1}{4}+\frac{\epsilon_j}{2}+\frac{i}{2}(y-t_j))}\right)\phi\left(\frac{y\log R}{2\pi}\right) dy\nonumber\\
&&=-\frac{1}{\pi}\int_{-\infty}^\infty \log (\pi) \phi(x)\frac{2\pi}{\log R} dx  \nonumber\\
&& \ \ +
\frac{1}{\pi}\int_{-\infty}^\infty\left(
\frac{\Gamma'(\frac{1}{4}+\frac{\epsilon_j}{2}+ \frac{it_j}{2}+\frac{i\pi}{\log R}x)}
{2\Gamma(\frac{1}{4}+\frac{\epsilon_j}{2}+ \frac{it_j}{2}+\frac{i\pi}{\log R}x)}+
\frac{\Gamma'(\frac{1}{4}+\frac{\epsilon_j}{2}- \frac{it_j}{2}+\frac{i\pi}{\log R}x)}
{2\Gamma(\frac{1}{4}+\frac{\epsilon_j}{2}- \frac{it_j}{2}+\frac{i\pi}{\log R}x)}\right)
\phi(x) \frac{2\pi}{\log R}dx. \nonumber\\ \label{eqn: gammaterm1}
\end{eqnarray}
The first integral in \eqref{eqn: gammaterm1} is $O(1/\log R)$ due to the rapid decay of $\phi$, and so does not contribute. Therefore, we just need to consider the second integral.
From equation (6.3.18) of \cite{AS}, we have for $|z|>1$ with $|\arg(z)| < \pi$ that
\be
\frac{\Gamma'(z)}{\Gamma(z)}\ = \ \log(z)-\frac{1}{2z} + \sum_{n=1}^\infty \frac{B_{2n}}{2nz^{2n}},
\ee
where the $B_{2n}$'s are the Bernoulli numbers. Thus, for $|z|>1$, we have
\be
\frac{\Gamma'(z)}{\Gamma(z)}\ =\ \log(z)+ O\left(\frac{1}{|z|}\right).
\ee
For the first term in the second integral in \eqref{eqn: gammaterm1}, we have
\be
z\ =\ \frac{1}{4}+\frac{\epsilon_j}{2}+\frac{it_j}{2}+\frac{i\pi}{\log R}x
\ee
with $t_j > 0$ and $x$ real. Thus we have
\begin{equation}\label{eq:withlogtoexpandlater}
\frac{\Gamma'(\frac{1}{4}+\frac{\epsilon_j}{2}+ \frac{it_j}{2}+\frac{i\pi}{\log R}x)}{ \Gamma(\frac{1}{4}+\frac{\epsilon_j}{2}+ \frac{it_j}{2}+\frac{i\pi}{\log R}x)}
= \log\left(\frac{1}{4}+\frac{\epsilon_j}{2}+ \frac{it_j}{2}+\frac{i\pi}{\log R}x\right) +
O\left(\frac{1}{|\frac{1}{4}+\frac{\epsilon_j}{2}+ \frac{it_j}{2}+\frac{i\pi}{\log R}x|}\right)
\end{equation}
for
\be
\left|\frac{1}{4}+\frac{\epsilon_j}{2}+\frac{it_j}{2}+\frac{i\pi}{\log R}x\right|\ >\ 1.
\ee
Since the only poles of the Gamma function are the non-positive integers and since our $z$ stays away from such integers, the denominators above are well-defined for our $t_j, x$. Expanding the logarithm in \eqref{eq:withlogtoexpandlater}, as $\epsilon_j \in \{0,1\}$ we get
\begin{eqnarray}
\frac{\Gamma'(\frac{1}{4}+\frac{\epsilon_j}{2}+ \frac{it_j}{2}+\frac{i\pi}{\log R}x)}{ \Gamma(\frac{1}{4}+\frac{\epsilon_j}{2}+ \frac{it_j}{2}+\frac{i\pi}{\log R}x)}
\ = \  \log(1+t_j) + O\left(\log(2+|x|) \right).
\end{eqnarray}

Thus
\begin{eqnarray}
&& \frac{1}{\pi}\int_{-\infty}^\infty
\frac{\Gamma'(\frac{1}{4}+\frac{\epsilon_j}{2}+ \frac{it_j}{2}+\frac{i\pi}{\log R}x)}
{2\Gamma(\frac{1}{4}+\frac{\epsilon_j}{2}+ \frac{it_j}{2}+\frac{i\pi}{\log R}x)} \phi(x) \frac{2\pi}{\log R}dx\nonumber\\
&&=\ \ \frac{1}{\log R}\int_{-\infty}^\infty
\left(\log\left(1+t_j\right) + O\left(\log(2+|x|)\right)  \right)  \phi(x) dx\nonumber\\
&&=\ \ \frac{1}{\log R}\int_{-\infty}^\infty
\log(1+t_j) \phi(x) dx
+O\left(\frac{1}{\log R}\right)\nonumber\\
&& =\ \ \widehat{\phi}(0)\frac{\log(1+t_j)}{\log R}+O\left(\frac{1}{\log R}\right).
\end{eqnarray}

The second term in \eqref{eqn: gammaterm1} is handled similarly; the only difference is that we have $-it_j$ instead of $it_j$, which is immaterial since the integrals are from $-\infty$ to $\infty$, $\phi$ is even and $|\Gamma(a+ib)| = |\Gamma(a+ib)|$.

Combining, we find the Gamma terms contribute
\begin{eqnarray}
&&
\frac{1}{2\pi i}\int_{(\frac{3}{2})}\left(-\log(\pi)+\frac{\Gamma'\left(\frac{s+\epsilon_j+it_j}{2}\right)}{2\Gamma\left(\frac{s+\epsilon_j+it_j}{2}\right)}
+\frac{\Gamma'\left(\frac{s+\epsilon_j-it_j}{2}\right)}{2\Gamma\left(\frac{s+\epsilon_j-it_j}{2}\right)}\right)[H(s)+H(1-s)]ds\nonumber\\
&& =\  2\widehat{\phi}(0) \frac{\log(1+t_j)}{\log R} + O\left(\frac{1}{\log R}\right)\nonumber\\
&& = \ \widehat{\phi}(0) \frac{\log(1+t_j^2)}{\log R} + O\left(\frac{1}{\log R}\right).
\end{eqnarray}
\end{proof}

\subsection{Analysis of the Satake parameters}

\begin{lem}\label{lem:expliciteqnnongamma}
The Satake parameters term in \eqref{eqn: explicit} contributes
\begin{eqnarray}\label{eqn: preliminary1level}
&& \frac{1}{2\pi i}\int_{(\frac{3}{2})}\left(\sum_p\log p\sum_{k=1}^\infty\frac{\alpha_j^k(p)+\beta_j^k(p)}{p^{sk}}\right)[H(s)+H(1-s)]ds \nonumber\\
&&= \ \sum_p\sum_{k=1}^\infty\frac{2(\alpha_j^k(p)+\beta_j^k(p))\log p}{p^{k/2}\log R}\widehat{\phi}\left(\frac{k\log p}{\log R}\right)\nonumber\\
&&= \ \sum_p \frac{2\lambda_j(p)\log p}{p^{1/2}\log R}\widehat{\phi}\left(\frac{\log p}{\log R}\right)\ +
\sum_p\frac{2(\lambda_j(p^2)-1)\log p}{p\log R}\widehat{\phi}\left(\frac{2\log p}{\log R}\right)\nonumber\\ & & \ \ \ \ \ \ \ \ + \ O\left(\frac{1}{\log R}\right).
\end{eqnarray}
\end{lem}
\begin{proof}
Using contour integration to shift integrals, we have the following equality:
\begin{eqnarray}\label{eqn: nongamma appendix}
&& \frac{1}{2\pi i}\int_{(\frac{3}{2})}\left(\sum_p\log p\sum_{k=1}^\infty\frac{\alpha_j^k(p)+\beta_j^k(p)}{p^{sk}}\right)[H(s)+H(1-s)]ds \nonumber\\
&& \ = \ \sum_p\sum_{k=1}^\infty\frac{2(\alpha_j^k(p)+\beta_j^k(p))\log p}{p^{k/2}\log R}\widehat{\phi}\left(\frac{k\log p}{\log R}\right).\end{eqnarray}
The proof of this equality is similar to that of other families;
see for example \cite{ILS,RS}.

We can truncate this sum at $k=2$ since $\log p \ll p^{\epsilon}$ for any $\epsilon$ and
$|\alpha_j(p)|, |\beta_j(p)| < p^{\delta}$, where $\delta < 7/64$ by work of Kim-Sarnak \cite{KSa} (note that this bound does not depend on $u_j$). Thus the contribution from $k\ge 3$ is
\begin{eqnarray}
\sum_p\sum_{k=3}^\infty\frac{2(\alpha_j^k(p)+\beta_j^k(p))\log p}{p^{k/2}\log R}\widehat{\phi}\left(\frac{k\log p}{\log R}\right) & \ \ll\ & \frac{1}{\log R}\sum_p\sum_{k=3}^\infty\frac{4 p^{k\delta} p^\epsilon}{p^{k/2}}\widehat{\phi}\left(\frac{k\log p}{\log R}\right) \nonumber\\
&\ \ll\ & \frac{1}{\log R}\sum_p\sum_{k=3}^\infty\frac{p^{k\delta} p^\epsilon}{p^{k/2}}, \label{eqn: nongamma2}
\end{eqnarray}
since $\widehat{\phi}$ has compact support. Since $\delta < 7/64$ (though all we really need is $\delta < 1/6$), then $1/2 -\delta > 1/3$, and we have
\begin{eqnarray}
\frac{1}{\log R}\sum_p\sum_{k=3}^\infty\frac{p^{k\delta} p^\epsilon}{p^{k/2}}
&=& \frac{1}{\log R}\sum_p \frac{p^\epsilon}{ p^{3(1/2-\delta)}} \frac{1}{1-1/p^{1/2-\delta}}\ =\ O\left(\frac{1}{\log R}\right).\ \ \ \ \label{eqn: nongamma3}
\end{eqnarray}

From \eqref{eqn: nongamma2} and \eqref{eqn: nongamma3}, we find
\begin{eqnarray}
&& \sum_p\sum_{k=1}^\infty\frac{2(\alpha_j^k(p)+\beta_j^k(p))\log p}{p^{k/2}\log R}\widehat{\phi}\left(\frac{k\log p}{\log R}\right)\nonumber\\
&& \ =\ \sum_p\sum_{k=1}^2\frac{2(\alpha_j^k(p)+\beta_j^k(p))\log p}{p^{k/2}\log R}\widehat{\phi}\left(\frac{k\log p}{\log R}\right)
+ O\left(\frac{1}{\log R}\right).
\end{eqnarray}
Noting $\alpha_j(p)+\beta_j(p) = \lambda_j(p)$ and $\alpha_j^2(p)+\beta_j^2(p) = \lambda_j(p^2)-1$  (see \eqref{eqn: product}) completes the proof.
\end{proof}


In \S\ref{sec: Kuznetsov}, we handle the $\lambda_j$ terms in \eqref{eqn: preliminary1level}
by averaging over the family of Maass forms by using the Kuznetsov formula.
However, we can consider the  $-2\log p/ p \log R$ term in the second sum
in \eqref{eqn: preliminary1level} without resorting to the Kuznetsov formula, as this term is independent of $u_j$.

\begin{lem}
The $u_j-$independent term in \eqref{eqn: preliminary1level} of Lemma \ref{lem:expliciteqnnongamma} contributes
\begin{equation}
\frac{2}{\log R}\sum_p\frac{\log p}{p}\widehat{\phi}\left(\frac{2 \log p}{\log R}\right)\ =\ \frac{\phi(0)}{2} +
O\left( \frac{\log\log R}{ \log R}\right).
\end{equation}
\end{lem}

The proof is standard and follows from the Prime Number Theorem and Partial Summation; see for instance Appendix C of \cite{Mil1} or chapter 15 of \cite{MT-B}.


\subsection{1-Level Density Expansion}

Combining the analysis of the Gamma and Satake parameter terms in the explicit formula, we've shown

\begin{prop}
The 1-level density for a single level 1 cuspidal Maass form equals
\begin{eqnarray}\label{eq:propexpansionD1}
D_1(u_j; \phi) &=& \widehat{\phi}(0) \frac{\log(1 + t_j^2)}{\log R }+
\frac{\phi(0)}{2} - \sum_p \frac{2\lambda_j(p)\log p}{p^{1/2}\log R}\widehat{\phi}\left(\frac{\log p}{\log R}\right)\nonumber\\
& & \ - \ \sum_p\frac{2\lambda_j(p^2)\log p}{p\log R}\widehat{\phi}\left(\frac{2\log p}{\log R}\right)
+O\left(\frac{\log\log R}{\log R}\right).\ \ \ \  \label{eqn: onefunction}
\end{eqnarray}
\end{prop}


Before we use this proposition to obtain the average weighted 1-level density, we first discuss the scaling constant $R$. Its purpose is to normalize the low-lying zeros for comparison between different families and with the random matrix ensembles. Specifically, we choose it so that we have a mean spacing of 1 near the central point, or the average of the first term in \eqref{eq:propexpansionD1} is $\widehat{\phi}(0)$. Explicitly, we need to pick $R$ such that
\begin{equation}\label{eqn: R}
\frac{1}{\sum_j \frac{h_T(t_j)}{\|u_j \|^2}} \sum_j \frac{h_T(t_j)}{\|u_j\|^2}\frac{\log(1 +t_j^2)}{\log R}\ =\ 1 + o(1).
\end{equation}
The formula clearly suggests that we take $R =T^2$. The proof requires the use of the Kuznetsov trace formula, and depends on a result to be proved in the next section; we give the rest of the argument here in order to be able to write down the final version of our 1-level density.\footnote{Actually, we'll see the proof for $h_{2,T}$ does not require the Kuznetsov trace formula, as the $|t_j|$ are tightly localized around $T \to\infty$ and thus there is no appreciable variation in the $\log(1+t_j^2)$ factors. The situation is different for $h_{1,T}$, as there $\log(1+t_j^2)$ varies greatly over the range. While with more careful book-keeping one may be able to avoid using the trace formula here, it does not seem worth the effort as we need to use the trace formula elsewhere in the proof.}

\begin{lem}\label{lem:findingR} Let $h_T$ be as in \eqref{eq:defht1} or \eqref{eq:defht2}. If $R = T^2$ and $\frac{\pi}{2\log T} < (1-\eta)L = o(T^{1/8})$, then \eqref{eqn: R} holds, with the little-oh term at most $O\left(1/\log R\right)$.
\end{lem}

Before proving the above lemma, we first state a needed result.

\begin{lem}\label{lem:findingRprep} Let $h_T$ be as in \eqref{eq:defht1} (i.e., $h_{1,T}$) or \eqref{eq:defht2} (i.e., $h_{2,T}$). Then\be\label{eqn: B} \twocase{\int_{-\infty}^\infty r \tanh(r)h_T(r)dr \ = \ }{T^2 + O(1)}{if $h_T = h_{1,T}$}{LT + O(1)}{if $h_T = h_{2,T}$.} \ee
\end{lem}

\begin{proof} We prove the second claim; the first follows similarly. Note that $\tanh(r)$ is odd and
\be
\tanh(r)\ =\  1 + O(e^{-2r})
\ee
for $r> 0$. As $h_T(r)$ is an even function and $\tanh(r)$ and $r$ are odd functions,
\begin{eqnarray}
\int_{-\infty}^\infty r \tanh(r)h_T(r)dr &=&  2\int_0^\infty r \tanh(r) h_T(r) dr \nonumber\\
&=& 2\int_0^\infty r \left(\frac12 h\left(\frac{r-T}{L}\right) + \frac12 h\left(\frac{r+T}{L}\right)\right)
(1+ O(e^{-2r}))  dr \nonumber\\
&=& 2\int_0^\infty r \left(\frac12 h\left(\frac{r-T}{L}\right) + \frac12 h\left(\frac{r+T}{L}\right)\right)dr + O\left( \int_0^\infty  re^{-2r}dr\right)\nonumber\\
& = & \int_{-T/L}^\infty (rL+T) h(r) Ldr + O(1) + O(1) \nonumber\\
&=& LT \int_{-\infty}^\infty h(r)dr + O(1) \ = \ LT + O(1),
r\end{eqnarray} where we used the facts that $h$ is an even Schwartz function (so $\int r h(r)dr = 0$), $h$ integrates to 1, and the relative sizes of $T$ and $L$.
\end{proof}

\begin{proof}[Proof of Lemma \ref{lem:findingR}] We first give the proof for $h_T = h_{2,T}$. Unfortunately we cannot just apply the Kuznetsov formula (see \eqref{eq:kuznetsovtraceformula}) with test function $h_{2,T}(t_j) \log (1+t_j^2)$, as our test function is supposed to have compact support (though we can argue along these lines for $h_{1,T}$, as we have weaker conditions there). Instead, we note that $h_{2,T}(t) = \frac12 h((t-T)/L) + \frac12 h((t+T)/L)$, with $\frac{\pi}{2\log T} < (1-\eta)L = o(T^{1/8})$ and $h$ a non-negative even Schwartz function. Thus the only $t_j$ that contribute are those near $\pm T$; we are quite safe if we only consider $\left||t_j| - T\right| \le LT^{1/2012}$, as the contribution from the $t_j$ further away from $\pm T$ is negligible and we have good control over the norms $||u_j||$ (see \eqref{eq:boundsnormuj}). For these restricted $t_j$, we have \be \log(1+t_j^2) \ = \ \log T^2 + O\left(L/T^{2010/2011}\right) \ = \ \log T^2 + O(T^{-6/8}).\ee We now remove our restriction on the $t_j$'s, and restore the sum to all values. We find that
\begin{eqnarray}\label{eqn: R1}
\sum_j \frac{h_T(t_j)}{\|u_j\|^2}\log(1 +t_j^2)
& \ =\ & \sum_j \frac{h_T(t_j)}{\|u_j\|^2}\left(\log(T^2) + O(T^{-3/4})\right)\nonumber\\ &=&
\frac1{\pi^2}\int_{-\infty}^\infty r \tanh(r) h_T(r) \log(T^2)dr + O(LT^{-1/4}) \nonumber\\ & & \ \ + \ O\left( L e^{\pi/2L} \log \left(T^{-1} e^{\pi/L}\right) \log (T^2)\right);
\end{eqnarray} here we used the analysis from the proof of Lemma \ref{lem: kuznetsovapproximation} below to bound two of the terms from the Kuznetsov formula. The proof is completed by using Lemma \ref{lem:findingRprep} to replace the integral with $LT\log(T^2)$. Taking $R=T^2$, we see equation \eqref{eqn: R} follows immediately (as the main term of the sum of the weights is $LT$). This completes the analysis for $h_{2,T}$.

We sketch the proof for $h_T = h_{1,T}$. We apply the Kuznetsov trace formula with weight function $h_{1,T}(t) \log(1+t^2)$. The analysis is similar. We can wait to expand the logarithm until the hyperbolic tangent integral term, which can be directly evaluated as $\tanh(r)$ is odd and equal to $1 + O(e^{-2r})$: \bea \int_{-\infty}^\infty r \tanh(r) h_{1,T}(r) \log (1+r^2)dr  & \ = \ & 2\int_0^\infty r \left(1 + O\left(e^{-2r}\right)\right) e^{-r^2/T^2} \log(1+r^2) dr \nonumber\\ &=& \int_0^\infty \log(1+u) e^{-u/T^2} du + O\left(\log^2 T\right) \nonumber\\ &=& \int_{\sqrt{\log T}}^{T^4} \left[\log \frac{u}{T^2} + \log(T^2)\right] e^{-u/T^2} + O\left(\log^2 T\right) \nonumber\\ &=& T^2 \log(T^2) + O(T^2).
\eea
Therefore
\begin{equation}\label{eqn: Rapprox2}
\frac{1}{\sum_j \frac{h_{1,T}(t_j)}{\|u_j \|^2}} \sum_j \frac{h_{1,T}(t_j)}{\|u_j\|^2}\frac{\log(1 +t_j^2)}{\log R}\  =\
\frac{\log T^2}{\log R} + O\left(\frac1{\log R}\right),
\end{equation}
and so if $R= T^2$, the sum in \eqref{eqn: Rapprox2} is 1 plus $O(1/\log R)$.
\end{proof}

\begin{rek} We do not need to use the Kuznetsov trace formula to derive \eqref{eqn: R1}, as all that matters is that up to a negligible error the left hand side is $\log(T^2)$ times the sum of our weights. We use the Kuznetsov trace formula here to keep the argument similar to the analysis for $h_{1,T}$, where we do need the trace formula. \end{rek}

Our above results immediately yield

\begin{lem}[1-Level Density Expansion]\label{lem:1levelexpansion} Let $h_T$ be as in \eqref{eq:defht1} or \eqref{eq:defht2}, and assume $\frac{\pi}{2\log T} < (1-\eta)L = o(T^{1/8})$. Let $\phi$ be an even Schwartz function whose Fourier transform has compact support. The weighted 1-level density of the family of level 1 cuspidal Maass forms is
\begin{eqnarray}
\mathcal{D}_1(\phi,h_T) & \ = \ & \frac{1}{\sum_j \frac{h_T(t_j)}{\|u_j\|^2}} \sum_j \frac{h_T(t_j)}{\|u_j\|^2} D_1(u_j; \phi) \nonumber\\
&=& \frac{\phi(0)}{2} + \widehat{\phi}(0)  \nonumber\\
& & \ -\ \frac{1}{\sum_j \frac{h_T(t_j)}{\|u_j\|^2}} \sum_p \frac{2\log p}{p^{1/2} \log R} \widehat{\phi}\left(\frac{\log p}{\log R}\right)
\sum_j \frac{h_T(t_j)}{\|u_j\|^2} \lambda_j(p)\nonumber\\
& &\ -\ \frac{1}{\sum_j \frac{h_T(t_j)}{\|u_j\|^2}} \sum_p \frac{2\log p}{p \log R} \widehat{\phi}\left(\frac{2\log p}{\log R}\right)
\sum_j \frac{h_T(t_j)}{\|u_j\|^2} \lambda_j(p^2)\nonumber\\ & & \ +\ O\left(\frac{\log\log R}{\log R}\right).\label{eqn: 1 level average preliminary}
\end{eqnarray}
\end{lem}

\begin{proof}
We can pull out the $O(\log \log R/ \log R)$ term from the sum over $u_j$ as the implied constant does not depend on $u_j$ and $h_T(t_j)/\|u_j\|^2$ is positive so that oscillation does not make this sum larger. We used Lemma \ref{lem:findingR} to replace $\frac{\widehat{\phi}(0)}{\sum_j \frac{h_T(t_j)}{\|u_j\|^2}} \sum_j \frac{h_T(t_j)}{\|u_j\|^2} \frac{\log(1 +t_j^2)}{\log R}$ with $\widehat{\phi}(0)$ plus errors subsumed by the other error terms. Finally, we can interchange the sum over $j$ and $p$ in the last two sums
in \eqref{eqn: 1 level average preliminary}  since the sum over $p$ is a finite sum as $\widehat{\phi}$ has compact support. \end{proof}

Therefore, to find the weighted average 1-level density, we need to evaluate the last two sums in \eqref{eqn: 1 level average preliminary}, which are all sums over the Fourier coefficients of the $u_j$'s. We do this in Sections \ref{sec: Kuznetsov} and \ref{sec:1level} by using the Kuznetsov trace formula for Maass cusp forms, which handles such sums.

\section{Kuznetsov Trace Formula}\label{sec: Kuznetsov}

\subsection{Expansion}

Recall that $\{u_j\}$ is a basis for the space of level 1 cuspidal Maass forms, and that the Fourier expansion of $u_j$ is
\be
u_j(z)\ =\ \cosh(t_j) \sum_{n\ne 0} \sqrt{y}\lambda_j(n) K_{iv}(2\pi |n|y)e^{2\pi i n x}.
\ee
We normalized $u_j$ by setting $\lambda_j(1)=1$.

Recall from \eqref{eq:kuznetsovtraceformula} that for such normalized $u_j$'s, the Kuznetsov trace formula states that
\begin{equation}
\sum_{j} \frac{h_T(t_j)}{\|u_j\|^2} \lambda_j(m)\overline{\lambda_j(n)}\ =\ A_{h_T}(m,n) + B_{h_T}(m,n)+ C_{h_T}(m,n),
\end{equation}
where
\begin{eqnarray}\label{eq:ABC}
A_{h_T}(m,n) &\ =\ & -\frac{1}{4\pi} \int_{-\infty}^\infty \overline{\tau(m,r)} \tau(n,r) \frac{h_T(r)}{\cosh( \pi r)}  dr\nonumber\\
B_{h_T}(m,n) &=& \frac{\delta_{n,m}}{\pi^2} \int_{-\infty}^\infty r \tanh(r)h_T(r)dr \nonumber\\
C_{h_T}(m,n) &=&
\frac{2i}{\pi} \sum_{c\ge 1} \frac{S(m,n;c)}{c}
\int_{-\infty}^\infty J_{2ir}\left( \frac{4\pi \sqrt{mn}}{c} \right) \frac{rh_T(r)}{\cosh(\pi r)} dr,
\end{eqnarray}
with $\tau, S, J_{2ir}$ defined in \eqref{eq:kuznetsovtraceformula}.
Since $\lambda_j(1) = 1$, we have the crucial observation that
\begin{equation}
\sum_{j} \frac{h_T(t_j)}{\|u_j\|^2} \lambda_j(m)\ =\ \sum_{j} \frac{h_T(t_j)}{\|u_j\|^2} \lambda_j(m)\overline{\lambda_j(1)},
\end{equation}
and so we can study the last three sums in \eqref{eqn: 1 level average preliminary} via the Kuznetsov formula.
To find these three sums in the cases required by Lemma \ref{lem:1levelexpansion}, we use the Kuznetsov formula in the cases $(m,n)$ equals $(1,1)$, $(p,1)$ and $(p^2,1)$ for the 1-level density, while for the 2-level density we must additionally consider $(p_1,p_2), (p_1^2,p_2)$ and $(p_1^2,p_2^2)$ (where $p_2$ may or may not equal $p_1$). We either use the weight function $h_{T}(t)$ from \eqref{eq:defht1}, which is even and is essentially a Gaussian at the origin, or from \eqref{eq:defht2}, which is even and has two bumps centered around $T$ and $-T$. We also take $L$ such that $\frac{\pi}{2\log T} < (1-\eta)L = o(T^{1/8})$.

\subsection{Approximating Terms in the Kuznetsov Trace Formula}

By approximating $A_{h_T}(m,n), B_{h_T}(m,n), C_{h_T}(m,n)$, we have the following lemma. In addition to determining the scaling for $R$, it allows us to execute the sums over the Fourier coefficients for the $n$-level densities, and is the key ingredient in the analysis. The hardest part of the argument is bounding the contribution from the $C_{h_T}(m,n)$ terms.

\begin{lem}\label{lem: kuznetsovapproximation}
For $m \in \{1,p_1,p_1^2\}$, $n \in \{1,p_2,p_2^2\}$, $h_T$ as in \eqref{eq:defht1} or \eqref{eq:defht2}, and $\frac{\pi}{2\log T} < (1-\eta)L = o(T^{1/8})$,
\begin{eqnarray}
& & \sum_j \frac{h_T(t_j)}{\|u_j\|^2} \lambda_j(m)\overline{\lambda_j(n)} \nonumber\\
& & \twocase{=\ }{\frac{\delta(m,n)T^2}{\pi^2} + O \left( (\log 3mn)^2 T^{7/4}(mn)^{1/4}  \right)}{if $h_T = h_{1,T}$}{\frac{\delta(m,n)LT}{\pi^2}  + O\left( L e^{\pi/2L} \log \left(T^{-1} e^{\pi/L}\right)  (mn)^{1/4} \log(3mn) \right)}{if $h_T = h_{2,T}$.}\nonumber\\
\end{eqnarray}
\end{lem}

\begin{proof}
We first consider $A_{h_T}(m,n)$. Note that we have the following equality for $\tau$:
\begin{eqnarray}
\overline{\tau(m,r)}\tau(n,r) &= & \pi\Gamma(1/2+ir)^{-1}\Gamma(1/2-ir)^{-1}\zeta(1+2ir)^{-1}\zeta(1-2ir)^{-1}(mn)^{-1/2} \nonumber\\ & & \ \
\sum_{ ab =|m| } \left( \frac{a}{b}\right)^{-ir} \sum_{ ab =|n| } \left( \frac{a}{b}\right)^{ir}
\nonumber\\
&=& \ \pi\frac{\sin(\pi(\frac{1}{2}+ir))}{\pi|\zeta(1+2ir)|^2 (mn)^{1/2}}
\sum_{ ab =|m| } \left( \frac{a}{b}\right)^{-ir} \sum_{ ab =|n| } \left( \frac{a}{b}\right)^{ir} \nonumber\\
&=&\ \frac{\cosh(\pi r)}{|\zeta(1+2ir)|^2 (mn)^{1/2}}
\sum_{ ab =|m| } \left( \frac{a}{b}\right)^{-ir} \sum_{ ab =|n| } \left( \frac{a}{b}\right)^{ir},
\end{eqnarray}
where we used the functional equation for $\Gamma$ in the second equality.

Because of our choices of $m$ and $n$, the sums over the divisors are bounded by 18.
Furthermore, $|\zeta(1+2ir)| \gg 1/\log(2+|r|)$ (see for example \cite{Liu}).
If $h_T = h_{2,T}$, then for any $\epsilon > 0$ we have
\begin{eqnarray}
A_{h_{2,T}}(m,n) &=& \frac{1}{4\pi}\int_{-\infty}^\infty \overline{\tau(m,r)} \tau(n,r) \frac{h_{2,T}(r)}{\cosh( \pi r)} dr\nonumber\\
&\ =\ & O\left(  \int_{0}^\infty\frac{\frac12 h\left(\frac{r-T}{L}\right) + \frac12 h\left(\frac{r+T}{L}\right)}{|\zeta(1+2ir)|^2 (mn)^{1/2} } dr \right)\nonumber\\
&=& O\left( \frac{1}{(mn)^{1/2}}\int_{0}^{\infty}\log(2+r) \cdot h\left(\frac{r-T}{L}\right) dr\right)\nonumber \\
& =& O\left( \frac{1}{(mn)^{1/2}}\int_{0}^{\infty}r^{\epsilon}h\left(\frac{r-T}{L}\right) dr\right)\nonumber \\
& =& O\left( \frac{L^\epsilon}{(mn)^{1/2}}\int_{0}^{\infty} \left|\frac{r-T}{L} + \frac{T}{L}\right|^\epsilon h\left(\frac{r-T}{L}\right) dr\right)\nonumber \\
&=& O\left( \frac{L^\epsilon}{(mn)^{1/2}}\int_{-T/L}^{\infty} \left[|u|^\epsilon + (T/L)^\epsilon\right] h(u) Ldu\right)\nonumber \\
&=& O\left(\frac{L^{1+\epsilon} + L T^{\epsilon}}{(mn)^{1/2}}\right) \ = \ O\left(\frac{L T^{\epsilon}}{(mn)^{1/2}}\right),
\end{eqnarray} where the rapid decay of $h$ implies $\int |u| h(u)du = O(1)$. A similar calculation shows \be A_{h_{2,T}}(m,n) \ = \  O\left(\frac{T^{1+\epsilon}}{(mn)^{1/2}}\right).\ee

For $B_{h_T}(m,n)$, in the proof of Lemma \ref{lem:findingRprep} we showed in \eqref{eqn: B} that \be\label{eqn:BB} \twocase{B_{h_T}(m,n) \ = \ \frac1{\pi^2}\int_{-\infty}^\infty r \tanh(r)h_T(r)dr \ = \ }{\frac{T^2 + O(1)}{\pi^2}}{if $h_T = h_{1,T}$}{\frac{LT+O(1)}{\pi^2}}{if $h_T = h_{2,T}$.} \ee

Finally, we consider the Bessel-Kloosterman piece $C_{h_T}(m,n)$, which is the most delicate part of the analysis and the heart of the proof. We first consider $h_{2,T}$, and argue as in \cite{Sar}. As this is a lower order term, there is no need to compute optimal constants or expansions. We first consider the integral term in \eqref{eq:ABC}. As the two summands from $h_{2,T}$ are handled similarly, we only consider the $h((r-T)/L)$ part. We are left with studying \bea I_c(L,T;m,n) & \ := \ & \int_{-\infty}^\infty J_{2ir}\left( \frac{4\pi \sqrt{mn}}{c} \right) \frac{rh\left(\frac{r-T}{L}\right)}{\cosh(\pi r)} dr \nonumber\\ & = & \int_{-\infty}^\infty J_{2i(Lr+T)}\left( \frac{4\pi \sqrt{mn}}{c} \right) \frac{(Lr+T)h\left(r\right)}{\cosh(\pi (Lr+T))} Ldr. \eea We have (see equation (9.1.10) of \cite{AS}) \be\label{eq:besselexpansionfromAS} J_{2ir}(z) \ = \ \frac{(z/2)^{2ir}}{\Gamma(1+2ir)} + \sum_{k=1}^\infty \frac{(z/2)^{2ir} (-z/4)^{2k}}{k!\Gamma(1+k+2ir)},\ee and by Stirling's formula \be \Gamma(z+1) \ \sim \ (z/e)^z z^{1/2} \sqrt{2\pi} \ee
(as we only care about the main term, this asymptotic expansion suffices, and simplifies some of the algebra). \\

We briefly summarize the proof. We insert our approximations into the integral, calculating the main and error terms. We eventually end up with the main term as a Fourier transform of $h$ evaluated at $\frac{L}{\pi} \log\left(\frac{cT}{\pi \sqrt{mn}}\right)$; as $\supp(\widehat{h} \subset (-1,1)$, this restricts which $c$ can contribute, and leads to a finite contribution from the resulting Kloosterman sum. Unfortunately, we need to be very careful in dealing with the error terms. The reason is that it is not enough to obtain that $I_c(L,T;m,n)$ is small relative to $T$; this is easy. The problem is we must have sufficient decay in $c$ so that the resulting $c$-sum converges (and we gain a factor of $c^{1/2+\epsilon}$ from using Weil's estimate for the Kloosterman sum). We accomplish this for the main term by interpreting the integral as evaluating the Fourier transform of $h$ outside $(-1,1)$ unless $c$ is small. \emph{In the analysis below, we ignore all error terms in order to highlight the argument. We finish the proof by showing how they can be handled in Appendix \ref{sec:boundingerrorsBK}, with their contribution bounded by the error from the main term if $c$ is small, or by a very small (in both $c$ and $T$) error if $c$ is large.}\\

A standard analysis shows that it suffices to keep just main term in the Bessel expansion (we can just repeat the following for each term and then observe the resulting sum is dominated by the first term), and the rapid decay in $h$ allows us to truncate the integral to $-\sqrt[4]{T}$ to $\sqrt[4]{T}$, and thus \bea\label{eq:ILTmnstart} I_c(L,T;m,n) & \ \sim \ & \int_{-\sqrt[4]{T}}^{\sqrt[4]{T}} \frac{(2\pi\sqrt{mn}/c)^{2i(Lr+T)}}{(2i(Lr+T)/e)^{2i(Lr+t)} (2i(Lr+T))^{1/2} \sqrt{2\pi}} \frac{2(Lr+T) h(r)Ldr}{e^{\pi(Lr+T)} + e^{-\pi(Lr+T)}}  \nonumber\\ &=& \sqrt{\frac1{\pi i}} \int_{-\sqrt[4]{T}}^{\sqrt[4]{T}} \left(\frac{\pi e\sqrt{mn}}{c(Lr+T)}\right)^{2i(Lr+T)} \frac{\sqrt{Lr+T}}{e^{-\pi(Lr+T)}} \frac{h(r) Ldr}{e^{\pi(Lr+T)} +
e^{-\pi(Lr+T)}} \nonumber\\ &=& \sqrt{\frac1{\pi i}} \int_{T-L\sqrt[4]{T}}^{T+L\sqrt[4]{T}}
\left(\frac{\pi e\sqrt{mn}}{cu}\right)^{2iu} u^{1/2} h\left(\frac{u-T}{L}\right) \frac{Ldu}{L}. \eea For $u \in [T-L\sqrt[4]{T}, T + L\sqrt[4]{T}]$, $|u-T| \le LT^{1/4}$, and thus we have \be\label{eq:expansionsqrtu} u^{1/2} \ = \ (T + u-T)^{1/2} \ = \ T^{1/2} \left(1 + \frac{u-T}{T}\right)^{1/2} \ = \ T^{1/2} + O\left(\frac{L}{T^{1/4}}\right) \ee and \bea\label{eq:expansionuu} u^{-2iu} & \ = \ & e^{-2i u \log(T+u-T)} \ = \ e^{-2iu \log T - 2iu \log\left(1 + \frac{u-T}{T}\right)}\nonumber\\ & \ = \ & e^{-2iu\log T - 2iu\frac{u-T}{T} + O(u(u-T)^2/T^2)} \nonumber\\ &=& e^{-2iu\log T - 2i(T+u-T)\frac{u-T}{T} + O(L^2/T^{1/2})} \nonumber\\ &=& e^{-2iu\log T -2iu+2iT + O(L^2/T^{1/4})} \nonumber\\ &=& e^{-2iu \log(eT)} e^{2iT} \left(1 + O\left(\frac{L^2}{T^{1/4}}\right)\right).\eea Substituting \eqref{eq:expansionsqrtu} and \eqref{eq:expansionuu} into our asymptotic \eqref{eq:ILTmnstart}, and ignoring for now the error terms, yields \bea I_c(L,T;m,n) & \ \sim \ & \frac{T^{1/2}e^{-2iT}}{\sqrt{\pi i}} \int_{T-L\sqrt[4]{T}}^{T+L\sqrt[4]{T}} h\left(\frac{u}{L}-\frac{T}{L}\right) e^{-2\pi iu \frac{1}{\pi}\log\left(\frac{cT}{\pi\sqrt{mn}}\right)} du \nonumber\\ &\sim & \frac{T^{1/2}e^{-2iT}}{\sqrt{\pi i}} \int_{-\infty}^\infty h\left(\frac{u}{L}-\frac{T}{L}\right) e^{-2\pi iu \frac{1}{\pi}\log\left(\frac{cT}{\pi\sqrt{mn}}\right)} du \nonumber\\ & = & \frac{T^{1/2}e^{-2iT} e^{-2\pi i \frac{T}{\pi} \log\left(\frac{cT}{\pi\sqrt{mn}}\right)}}{\sqrt{\pi i}} L \widehat{h}\left(\frac{L}{\pi} \log\left(\frac{cT}{\pi\sqrt{mn}}\right)\right) \eea by the Fourier Transform identity \be \int_{-\infty}^\infty h(au+b) e^{-2\pi i uy} du \ = \ e^{2\pi i \frac{b}{a}y} \frac1{|a|} \widehat{h}\left(\frac{y}{a}\right).\ee We have therefore shown that \be |I_c(L,T;m,n)| \ \ll \ LT^{1/2} \left|\widehat{h}\left(\frac{L}{\pi} \log\left(\frac{cT}{\pi\sqrt{mn}}\right)\right)\right|.\ee As $\widehat{h}(x)$ vanishes if $|x| > 1$, the only $c$ that contribute are those with $\frac{L}{\pi} \log\left(\frac{cT}{\pi\sqrt{mn}}\right) \le 1$, or $c \le \pi \sqrt{mn} T^{-1}e^{\pi/L}$. This gives us the main term of the contribution from the Bessel term; we show the errors introduced by our approximations are negligible in Appendix \ref{sec:boundingerrorsBK}.

We substitute this bound and $c$-restriction into the expansion for $C_{h_T}(m,n)$ and use Weil's bound for the Kloosterman sum ($S(m,n;c) \ll \gcd(m,n,c)^{1/2}  \tau(c) c^{1/2}$, where $\tau(c) \ll c^\epsilon$ is the divisor function). For the 1-level sums, $n=1$ and writing $p$ for $p_1$ we have $m \in \{1,p,p^2\}$, while for the 2-level sums $m \in \{1,p_1,p_1^2\}$ and $n \in \{1,p_2,p_2^2\}$. We first handle the case when $n=1$, which includes all the sums that would arise in the 1-level density and some that occur in the 2-level; note this implies $\gcd(m,n,c)=1$. Using standard bounds for sums of the divisor function\footnote{We use $\sum_{c\le x} \tau(c)/c^{1/2} \ll x^{1/2} \log x$, which follows from partial summation and standard bounds for sums of divisor functions; see Chapter 1 of \cite{IK}.} yields\bea\label{eq:approxChTneeded} C_{h_T}(m,n) &\ = \ & \frac{2i}{\pi} \sum_{c\ge 1} \frac{S(m,n;c)}{c} \int_{-\infty}^\infty J_{2ir}\left( \frac{4\pi \sqrt{mn}}{c} \right) \frac{rh_T(r)}{\cosh(\pi r)} dr \nonumber\\ &\ll& \sum_{c \ge 1} \frac{\tau(c) \gcd(m,n,c)^{1/2}}{c^{1/2}} LT^{1/2} \left|\widehat{h}\left(\frac{L}{\pi} \log\left(\frac{cT}{\pi\sqrt{mn}}\right)\right)\right| \nonumber\\ &\ll & LT^{1/2} \sum_{c = 1}^{\pi \sqrt{mn} T^{-1}e^{\pi/L}} \frac{\tau(c)}{c^{1/2}}   \nonumber\\ &\ll& LT^{1/2} \left(\pi \sqrt{mn} T^{-1}e^{\pi/L}\right)^{1/2} \log\left(\pi \sqrt{mn} T^{-1}e^{\pi/L}\right) \nonumber\\ & \ll & L e^{\pi/2L} \log \left(T^{-1} e^{\pi/L}\right)  (mn)^{1/4}  \log(3mn).  \eea We want this to be smaller than the main term, which is $LT$ and arises from the hyperbolic tangent integral; thus $L$ cannot be too small. A close to optimal choice is to take $\frac{\pi}{2\log T} < (1-\eta)L$, so $e^{\pi/2L} < T^{1-\eta}$. 

We now consider the case when $\gcd(m,n) = p$. There are two sub-cases: $p|c$ and $p\notdiv c$. The contribution from the terms where $p\notdiv c$ is identical to the previous argument, as in that case $\gcd(m,n,c)=1$. We now handle the terms where $p|c$. We re-write $c$ as $pc'$, and note $\gcd(m,n,c)^{1/2}/c^{1/2}  = 1/c'^{1/2}$, and $\tau(pc') \ll \tau(c')$. We thus have the same sum as before, except now $c'$ only goes up to $1/p$ what $c$ did, and thus the error term here is subsumed in the previous error term. The argument is similar if $\gcd(m,n)=p^2$, and we again obtain the same error term. This completes the analysis of $C_{h_T}$ when $h=h_{2,T}$.



Unfortunately, a similar calculation fails for $h_T = h_{1,T}$. The difficulty is that the resulting term is now \be 2 {\rm Im} \int_0^\infty J_{2ir}\left(\frac{4\pi\sqrt{mn}}{c}\right) \frac{r h(r/T)}{\cosh(\pi r)}. \ee The problem is that for $T$ large, $h$ is essentially 1 in the integrand, and we lose the $T$-decay. Trying to use integral formulations of the Bessel and Gamma functions do not lead to tractable sums for bounding purposes. We thus resort to a different argument. We instead use equation (16.56) of \cite{IK}, which gives \be
C_{h_T}(m,n)\ =\ O\left((\log 3mn)^2 T^{7/4}(mn)^{1/4}\right).
\ee

By combining the terms $A_{h_T}(m,n)$, $B_{h_T}(m,n)$, $C_{h_T}(m,n)$ and noting that $A_{h_T}(m,n)$ $\ll$ $C_{h_T}(m,n)$, we obtain the desired result.
\end{proof}

\section{1-level Density}\label{sec:1level}

To determine the weighted 1-level density, we need to evaluate the last two sums in \eqref{eqn: 1 level average preliminary}, which we do in the following lemma.

\begin{lem}\label{lem:twosumsfor1leveldensity} Let $h_T$ be as in \eqref{eq:defht2}, $R = T^2$ and $T/L \gg \log T$.
Then
\begin{eqnarray}\label{eqn: twosums}
\frac{1}{\sum_j \frac{h_T(t_j)}{\|u_j\|^2}} \sum_p \frac{2\log p}{p^{1/2} \log R} \widehat{\phi}\left(\frac{\log p}{\log R}\right)
\sum_j \frac{h_T(t_j)}{\|u_j\|^2} \lambda_j(p)
&=& o(1)\nonumber\\\nonumber\\
\frac{1}{\sum_j \frac{h_T(t_j)}{\|u_j\|^2}} \sum_p \frac{2\log p}{p \log R} \widehat{\phi}\left(\frac{2\log p}{\log R}\right)
\sum_j \frac{h_T(t_j)}{\|u_j\|^2} \lambda_j(p^2)
&=& o(1).
\end{eqnarray}
\end{lem}

\begin{proof}
We first determine $\sum_j \frac{h_T(t_j)}{\|u_j\|^2}$, which is used to normalize the weights in \eqref{eqn: twosums} by having them sum to 1. As $\lambda_j(1) =1$, we have by
Lemma \ref{lem: kuznetsovapproximation} with $m = n =1$ that
\begin{eqnarray}
\sum_j  \frac{ h_T(t_j )}{\|u_j \|^2} & \ = \ & \sum_j  \frac{h_T(t_j ) }{ \|u_j \|^2} \lambda_j(1)\overline{\lambda_j(1)} \nonumber\\ & = &  \twocase{}{T^2/\pi^2+O(T^{7/4})}{if $h_T = h_{1,T}$}{LT/\pi^2 + O(L e^{\pi/2L} \log \left(T^{-1} e^{\pi/L}\right)}{if $h_T = h_{2,T}$.}
\end{eqnarray}
Therefore the reciprocal of this sum  is
\begin{eqnarray}
\twocase{\frac{1}{\sum_j  \frac{h_T(t_j )}{ \|u_j \|^2}} \ = \ }{\pi^2/T^2 + O\left(1/T^{9/4}\right)}{if $h_T=h_{1,T}$}{\pi^2/LT + O\left((LT)^{-1}T^{-1} e^{\pi/2L} \log \left(T^{-1} e^{\pi/L}\right)\right)}{if $h_T=h_{2,T}$.}
\end{eqnarray}

We first do the case when $h_T = h_{2,T}$ and then discuss the minimal changes needed if $h_T = h_{1,T}$.

Now we consider the first sum in \eqref{eqn: twosums}. By Lemma \ref{lem: kuznetsovapproximation} with $m = p, n =1$, we have
\begin{eqnarray}
\sum_j  \frac{h_T(t_j)}{\|u_j \|^2} \lambda_j(p) \ = \ \sum_j \frac{h_T(t_j)}{\|u_j \|^2} \lambda_j(p) \overline{\lambda_j(1)}\ = \  O(L e^{\pi/2L}\log T \cdot p^{1/4+\epsilon}).
\end{eqnarray}
As $R = T^2$, $\frac{\pi}{2\log T} < (1-\eta)L = o(T^{1/8})$ and ${\rm supp}(\widehat{\phi}) \subset (-\sigma, \sigma)$, the prime sum in \eqref{eqn: twosums} is over $p\le T^{2\sigma}$ and
\begin{eqnarray}
&& \frac{1}{\sum_j \frac{h_T(t_j)}{\|u_j \|^2}} \sum_{p\le T^{2\sigma}} \frac{2\log p}{p^{1/2} \log R} \widehat{\phi}\left(\frac{\log p}{\log R}\right)
\sum_j \frac{h_T(t_j)}{||u_j||^2} \lambda_j(p)\nonumber\\
&&\ \ \ \ =\ O\left(\frac{1}{LT}\right) \sum_{p\le T^{2\sigma}} \frac{2\log p}{p^{1/2} \log R} \left|\widehat{\phi}\left(\frac{\log p}{\log R}\right)\right|
O\left(LT^{1-\eta}\log T \cdot  p^{1/4+\epsilon} \right)\nonumber\\
&& \ \ \ \ =\ O \left(\frac{\log T}{T^{\eta}}\sum_{p\le T^{2\sigma}} p^{-1/4+\epsilon}\right) \nonumber\\
&& \ \ \ \ = \ O(T^{2\sigma(3/4 + \epsilon) -\eta} \log T).
\end{eqnarray} Note this sum is negligible so long as $\sigma < \frac23 \eta$ (and $\eta < 1$), so by taking $\eta$ arbitrarily close to $1$ we can get $\sigma$ arbitrarily close to $\frac23$. We could also obtain such support by taking $L$ even larger; for example, if $L \ge \frac{\pi \log\log T}{2\log T}$ then $e^{pi/2L} \le T^{1/\log\log T} \ll T^\epsilon$ for any $\epsilon > 0$.

The second sum in \eqref{eqn: twosums} may be handled similarly.
Using  $m=p^2 , n =1$ in Lemma \ref{lem: kuznetsovapproximation}, for any $\epsilon >0$ we have
\begin{eqnarray}
\sum_j  \frac{h_T(t_j)}{\|u_j \|^2}\lambda_j(p^2)  = \sum_j  \frac{h_T(t_j)}{\|u_j \|^2}\lambda_j(p^2)\overline{\lambda_j(1)} =  O(L e^{\pi/2L}\log T \cdot p^{1/4+\epsilon} \cdot p^{1/2+\epsilon}).
\end{eqnarray} As $R = T^2$ and ${\rm supp}(\widehat{\phi}) \subset (-\sigma, \sigma)$, this time the prime sum is over $p \le T$ and we have
\begin{eqnarray}
&& \frac{1}{\sum_j \frac{h_T(t_j)}{\|u_j \|^2}} \sum_p \frac{2\log p}{p \log R} \widehat{\phi}\left(\frac{2\log p}{\log R}\right)
\sum_j \frac{h_T(t_j)}{\|u_j \|^2} \lambda_j(p^2) \nonumber\\
&&\ \ \ \ =\ O\left(\frac{1}{LT}\right)
\sum_{p\le T^{\sigma}} \frac{2\log p}{p \log R} \widehat{\phi}\left(\frac{2\log p}{\log R}\right)
O\left(L e^{\pi/2L}\log T \cdot p^{1/2+\epsilon}\right) \nonumber\\
&&\ \ \ \ =\ O\left ( T^{-\eta} \log T \sum_{p\le T^{\sigma}} p^{-1/2+\epsilon} \right)\nonumber\\
&&\ \ \ \ =\  O\left(T^{\sigma(1/2+\epsilon) - \eta}\right),
\end{eqnarray} and thus this term is negligible for $\sigma < \frac{\eta}{1/2+\epsilon}$. Not surprisingly, the support restriction is weaker here than in the previous sum (we sum over fewer primes and divide by a higher power of $p$); once $\eta > 1/2$ this term does not contribute for support less than 1.

If now $h_T = h_{1,T}$, then the arguments above are trivially changed; we simply need to replace $L e^{\pi/2L} \log T$ with $T^{7/4}$, and the size of the family is now of the order $T^2$ instead of $LT$. The net effect is to replace $T^{-\eta}$ with $T^{-1/4}$ (up to a factor of $\log T$, which is immaterial). The main error term is now $O(T^{2\sigma(3/4+\epsilon)-1/4})$, which is negligible if $\sigma < 1/6$.\end{proof}

We can now prove Theorem \ref{thm:levelone} and determine the 1-level density.

\begin{proof}[Proof of Theorem \ref{thm:levelone}]
The proof follows immediately by substituting the results of Lemma \ref{lem:twosumsfor1leveldensity} into the 1-level density expansion of Lemma \ref{lem:1levelexpansion}.
\end{proof}

\section{2-Level Density}\label{sec:2level}

Miller \cite{Mil0,Mil1} noticed that while the 1-level density is unable to distinguish the three orthogonal groups if the Fourier transform of the test function is supported in $(-1,1)$ (though it can distinguish these from unitary and symplectic), the 2-level density is different for each of the classical compact groups for arbitrarily small support. The difference between the three orthogonal groups is entirely due to the distribution of signs of the functional equations. Thus, in order to determine which orthogonal group corresponds to our family, we now compute the 2-level density. As our goal is simply to uniquely identify which orthogonal group can correspond to our family, to simplify the exposition we simply concentrate on obtaining a small window of support about zero; with a very small amount of additional work one could obtain explicit bounds on the support.

The weighted 2-level density is
\begin{equation}
\mathcal{D}_{2}(\phi, h_T)\ :=\
\frac{1}{\sum_{j} \frac{h_T(t_j)}{\| u_j \|^2}}
\sum_{j} \frac{h_T(t_j)}{\| u_j \|^2}
 \sum_{j_1 \ne \pm j_2}
\phi_1 \left(\frac{\log R}{2\pi}\gamma_{j_1} \right)
\phi_2 \left(\frac{\log R}{2\pi}\gamma_{j_2} \right),
\end{equation}
where $\phi(x,y) = \phi_1(x) \phi_2(y)$ and $\widehat{\phi_1}, \widehat{\phi_2}$ are both supported in $(-\sigma, \sigma)$. The analysis is simplified by adding back the $j_1 = \pm j_2$ terms and then subtracting these off. This allows us to use the explicit formula twice for the sum over $j_1$ and $j_2$, while for the subtracted off term (from $j_1 = \pm j_2$), we essentially have a 1-level density. The only problem is that if the functional equation is odd, then there should only be one $j_2$ corresponding to $j_1=0$ and not 2, and thus we need to add back this contribution. Defining \be
\mathcal{N}(-1)\ :=\ \frac{1}{\sum_{j} \frac{h_T(t_j)}{\| u_j \|^2}}
 \sum_{j, \epsilon_j = -1} \frac{h_T(t_j)}{\| u_j \|^2}
\ee (the weighted percentage of Maass forms in our family with odd functional equation), we consider the modified version of the 2-level density (where we allow the $j_1 = \pm j_2$ terms)
\begin{eqnarray}
\mathcal{D}^*_{2}(\phi,h_T) &\ := \ &   \frac{1}{\sum_{j} \frac{h_T(t_j)}{\| u_j \|^2}}
 \sum_{j} \frac{h_T(t_j)}{\| u_j \|^2}   D_1(u_j,\phi_1) \overline{D_1(u_j,\phi_2)},\label{eqn: D*}
\end{eqnarray}
where
\begin{eqnarray}
D_1(u_j,\phi_i) &\ = \ &
\frac{\phi_i(0)}{2}+
\widehat{\phi_i}(0) \frac{\log(1 + t_j^2)}{\log R } + O\left(\frac{\log\log R}{\log R}\right) \nonumber\\
&&\ - \ \sum_p \frac{2\lambda_j(p)\log p}{p^{1/2}\log R}\widehat{\phi}_i \left(\frac{\log p}{\log R}\right) -
\sum_p\frac{2\lambda_j(p^2)\log p}{p\log R}\widehat{\phi}_i \left(\frac{2\log p}{\log R} \right) \nonumber\\ &=& \frac{\phi_i(0)}{2}+
\widehat{\phi_i}(0) \frac{\log(1 + t_j^2)}{\log R } + O\left(\frac{\log\log R}{\log R}\right)\nonumber\\ & &\ -\ S_1(u_j,\phi_i) - S_2(u_j,\phi_i)
\end{eqnarray}
with \begin{eqnarray}
S_{1}(u_j,\phi_i) &\ := \ & \sum_p \frac{2\lambda_j(p)\log p}{p^{1/2}\log R}\widehat{\phi}_i \left(\frac{\log p}{\log R}\right)\nonumber\\
S_{2}(u_j,\phi_i) &\ := \ & \sum_p\frac{2\lambda_j(p^2)\log p}{p\log R}\widehat{\phi}_i \left(\frac{2\log p}{\log R} \right).
\end{eqnarray}
Note that in \eqref{eqn: D*}, we can have the complex conjugate of $D_1(u_j,\phi_2)$  since $\phi$ is real. We do this so that we can apply
the Kuznetsov formula, which has a complex conjugate over one of the $\lambda_j$'s.

We can obtain the 2-level density from the modified 2-level density by subtracting off the contribution from $j_1 = \pm j_2$, which is
\begin{eqnarray}
\mathcal{D}_{2,\pm}(\phi,h_T) &\ := \ & 2 \mathcal{D}_1(\phi_1\phi_2,h_T) - (\phi_1 \phi_2)(0)
\frac{1}{\sum_{j} \frac{h_T(t_j)}{\| u_j \|^2}}
 \sum_{j, \epsilon_j = -1} \frac{h_T(t_j)}{\| u_j \|^2};
\end{eqnarray} note here that the test function is $\phi_1(u)\phi_2(u)$, and the last term on the right-hand-side is the weighted percentage of Maass forms with odd sign in functional equation, which we denote by $\mathcal{N}(-1)$. We have

\begin{lem}
For $h_T = h_{2,T}$, if $\frac{\pi}{2\log T} < (1-\eta)L = o(T^{1/8})$ and $\sigma < \frac{1}{3}\eta$, the weighted 2-level density is
\bea\label{eq:firstexpansion2levelfull} \mathcal{D}_2(\phi,h_T) & = & \frac{1}{\sum_{j} \frac{h_T(t_j)}{\| u_j \|^2}}  \sum_{j} \frac{h_T(t_j)}{\| u_j \|^2} \prod_{i=1}^2\Bigg[\widehat{\phi_i}(0) \frac{\log(1 + t_j^2)}{\log R }   - S_1^{c(i)}(u_j,\phi_i) - S_2^{c(i)}(u_j,\phi_i)\Bigg]^2 \nonumber\\ & & \ \ -\ \mathcal{D}_1(\phi_1\phi_2,h_T) + (\phi_1 \phi_2)(0) \mathcal{N}(-1) + O\left(\frac{\log\log R}{\log R}\right), \eea where $c(1)$ is the identity map and $c(2)$ denotes complex conjugation. Taking $\eta$ to be arbitrarily close to $1$, we get the result for $\sigma < 1/3$.

For $h_T = h_{1,T}$, the result holds for $ \sigma < 1/12$.
\end{lem}

\begin{proof} Our arguments above proved \eqref{eq:firstexpansion2levelfull}, except for the presence of a $O(\log\log R/\log R)$ inside the product over $i$. We now show that term may be removed from the product at the cost of an error (of the same size) outside all the summations.

As the $O(\log\log R/\log R)$ is independent of $u_j$, these terms are readily bounded by applying the Cauchy-Schwarz inequality. Letting $\mathcal{S}$ represent any of the factors in the product over $i$ in \eqref{eq:firstexpansion2levelfull}, the product involving this is $O(\log\log R/\log R)$:
\begin{eqnarray}& &
\frac{1}{\sum_{j} \frac{h_T(t_j)}{\| u_j \|^2}} \sum_{j} \frac{h_T(t_j)}{\| u_j \|^2} \mathcal{S}\cdot O\left(\frac{\log \log R}{\log R}\right)\nonumber\\ & &  \ll \
\left[\frac{1}{\sum_{j} \frac{h_T(t_j)}{\| u_j \|^2}}\sum_{j} \frac{h_T(t_j)}{\| u_j \|^2} O\left(\left(\frac{\log \log R}{\log R}\right)^2\right)\right]^{1/2}
\cdot
\left[\frac{1}{\sum_{j} \frac{h_T(t_j)}{\| u_j \|^2}}\sum_{j} \frac{h_T(t_j)}{\| u_j \|^2} |\mathcal{S}|^2\right]^{1/2}\nonumber\\
&& \ll \  O\left(\frac{\log \log R}{\log R}\right) \cdot
\left(\frac{1}{\sum_{j} \frac{h_T(t_j)}{\| u_j \|^2}} \sum_{j} \frac{h_T(t_j)}{\| u_j \|^2} |\mathcal{S}|^2\right)^{1/2}. \label{eqn: Cauchy}
\end{eqnarray}

We now analyze the four possibilities for the sum involving $|\mathcal{S}|^2$. If $\mathcal{S}$ is either $\widehat{\phi_i}(0) \frac{\log(1 + t_j^2)}{\log R }$ or $O\left(\frac{\log\log R}{\log R}\right)$, this sum is trivially $O(1)$, and thus the entire expression is $O\left(\frac{\log\log R}{\log R}\right)$. We are left with the non-trivial cases of $\mathcal{S} = S_1 $ or $\mathcal{S} = S_2$. To handle these cases, we rely on results that we will soon prove: for
$h_T = h_{2,T}$, if $3\sigma < \eta$, from Lemma \ref{lem: S1S1} we have $|\mathcal{S}_1|^2 = O(1)$ and from Lemma \ref{lem: S1S2} we have $|S_2|^2 = o(1)$. Note that we use these lemmas for $\phi = \phi_1 \phi_1$ instead of the usual $\phi = \phi_1 \phi_2$, but this does not affect the proofs. Similarly, for
$h_T = h_{1,T}$, these statements hold if $\sigma < 1/12$.
\end{proof}

By symmetry, it suffices to analyze the following terms to determine the 2-level density:\begin{eqnarray}
&& \frac{1}{\sum_{j} \frac{h_T(t_j)}{\| u_j \|^2}} \sum_{j} \frac{h_T(t_j)}{\| u_j \|^2}
\widehat{\phi_1}(0)\widehat{\phi_2}(0) \left(\frac{\log (1 + t_j^2)}{\log R}\right)^2 \nonumber\\
&&\frac{1}{\sum_{j} \frac{h_T(t_j)}{\| u_j \|^2}} \sum_{j} \frac{h_T(t_j)}{\| u_j \|^2}
\widehat{\phi_1}(0) \frac{\log (1 + t_j^2)}{\log R} S_k(u_j,\phi_2), \ \ k \in \{1, 2\}\nonumber\\
&&\frac{1}{\sum_{j} \frac{h_T(t_j)}{\| u_j \|^2}} \sum_j \frac{h_T(t_j)}{\| u_j \|^2}
S_k(u_j,\phi_1)\overline{S_\ell(u_j,\phi_2)} \ \ k, \ell \in \{1, 2\}.
\end{eqnarray}

We now analyze these terms.

\begin{lem} For $h_{T}$ as in \eqref{eq:defht1} or \eqref{eq:defht2}, $R=T^2$ and $\frac{\pi}{2\log T} < (1-\eta)L = o(T^{1/8})$, we have
\be
\frac{1}{\sum_{j} \frac{h_T(t_j)}{\| u_j \|^2}}\sum_{j} \frac{h_T(t_j)}{\| u_j \|^2}
\widehat{\phi_1}(0) \widehat{\phi_2}(0) \left(\frac{\log(1 + t_j^2)}{\log R} \right)^2
= \widehat{\phi_1}(0) \widehat{\phi_2}(0) + O\left(\frac1{\log\log R}\right).
\ee

\end{lem}

\begin{proof} The proof is analogous to the proof of Lemma \ref{lem:findingR}.
\end{proof}

\begin{lem} If $R=T^2$, $\frac{\pi}{2\log T} < (1-\eta)L = o(T^{1/8})$, and $\sigma<\frac{2}{3}\eta$ (for $h_{2,T}$) or $\sigma<\frac{1}{6}$ (for $h_{1,T}$), the $\widehat{\phi_1}(0) S_1(u_j,\phi_2)$ and $\widehat{\phi_1}(0) S_2(u_j,\phi_2)$ terms are $O(\log\log R/\log R)$, and thus do not contribute.
\end{lem}

\begin{proof} The proof is almost identical to the application of the Kuznetsov trace formula to prove similar results for the 1-level density, the only change being that now we have the modified weight function $h_T(t_j) \log (1+t_j^2)$. \end{proof}

\begin{lem}\label{lem: S1S1}
For $ h_T = h_{2,T}$, if $\frac{\pi}{2\log T} < (1-\eta)L = o(T^{1/8})$ and $ \sigma < \frac{1}{3}\eta$,  then
\begin{equation}
\frac{1}{\sum_j\frac{h_T(t_j)}{\|u_j\|^2}} \sum_j\frac{h_T(t_j)}{\|u_j\|^2}S_1(u_j,\phi_1)\overline{S_1(u_j,\phi_2)}
= 2 \int_{-\infty }^\infty  |z|\widehat{\phi}_1(z)\widehat{\phi}_2(z)dz+ O\left(\frac{\log \log R}{\log R}\right).
\end{equation}
Taking $\eta$ to be arbitrarily close to $1$, the statement holds for $\sigma < 1/3$.

For $h_T =  h_{1,T}$, the result holds if $\sigma < 1/12$.
\end{lem}

\begin{proof} We give first give proof for $h_T = h_{2,T}$.
\begin{eqnarray}
&&\sum_j\frac{h_T(t_j)}{\|u_j\|^2}S_1(u_j,\phi_1)\overline{S_1(u_j,\phi_2)}\nonumber\\
&&= \ 4\sum_j\frac{h_T(t_j)}{\|u_j\|^2}\sum_{p_1,p_2}\frac{\lambda_j(p_1)\overline{\lambda_j(p_2)}}{p_1^{1/2}p_2^{1/2}}\frac{\log p_1\log p_2}{\log^2 R}\widehat{\phi}_{1}\left(\frac{\log p_1}{\log R}\right)\widehat{\phi}_{2}\left(\frac{\log p_2}{\log R}\right)\nonumber\\
&&= \ 4\sum_{p_1,p_2}\frac{1}{p_1^{1/2}p_2^{1/2}}\frac{\log p_1\log p_2}{\log^2 R}\widehat{\phi}_{1}\left(\frac{\log p_1}{\log R}\right)\widehat{\phi}_{2}\left(\frac{\log p_2}{\log R}\right)\sum_j\frac{h_T(t_j)}{\|u_j\|^2}\lambda_j(p_1)\overline{\lambda_j(p_2)}. \nonumber\\
\end{eqnarray}
As before, we apply the Kuznetsov formula to the inner sum. Since the formula has a $\delta_{p_1, p_2}$, we need to split this sum into the case when
$p_1= p_2$ and the case when $p_1 \ne p_2$. We first show that the $p_1 \ne p_2$ does not contribute. As $R=T^2$, the prime sums are over distinct primes at most $T^{2\sigma}$. If $p_1 \neq p_2$, then using Lemma \ref{lem: kuznetsovapproximation} for $h_T = h_{2,T}$, we have
\begin{eqnarray}
&&4\sum_{p_1 \ne p_2}\frac{1}{p_1^{1/2}p_2^{1/2}}\frac{\log p_1\log p_2}{\log^2 R}\widehat{\phi}_{1}\left(\frac{\log p_1}{\log R}\right)\widehat{\phi}_{2}\left(\frac{\log p_2}{\log R}\right)\sum_j\frac{h_T(t_j)}{\|u_j\|^2}\lambda_j(p_1)\overline{\lambda_j(p_2)}\nonumber\\
&&= \ 4\sum_{p_1 \ne p_2}\frac{\log p_1\log p_2}{p_1^{1/2}p_2^{1/2}\log^2 R}\widehat{\phi}_{1}\left(\frac{\log p_1}{\log R}\right)\widehat{\phi}_{2}\left(\frac{\log p_2}{\log R}\right)O\left( (\log p_1 p_2)^2 Le^{\pi/2L} \log T \cdot (p_1 p_2)^{1/4}\right) \nonumber\\
&&=\ O\left(Le^{\pi/2L} \log T \sum_{p_1 \ne p_2 < T^{2\sigma}}\frac{1}{p_1^{1/4}p_2^{1/4}}\frac{\log^3 p_1\log^3 p_2}{\log^2 R} \right)\nonumber\\
&&= \ O\left(Le^{\pi/2L} \log T
\left(\sum_{p_1 < T^{2\sigma}}\frac{1}{p_1^{1/4}}\frac{\log^3 p_1}{\log R} \right)
\left(\sum_{p_2 < T^{2\sigma}}\frac{1}{p_2^{1/4}}\frac{\log^3 p_2}{\log R} \right)\right)\nonumber\\
&&= \ O\left(Le^{\pi/2L} \log T \cdot T^{2\sigma(3/4 + \epsilon)}  T^{2\sigma(3/4 + \epsilon)}\right).
\end{eqnarray}
Since we are dividing by $LT + o(1)$ and $\frac{\pi}{2\log T} < (1-\eta)L$ (so the $e^{2/\pi L}$ term is bounded by $T^{1-\eta}$),
this term does not contribute if $2\sigma(3/4) + 2\sigma(3/4) < \eta$. As before, we can take $\eta$ to be arbitrarily less than $1$, and so the term does not contribute if $\sigma < 1/3$.

The case $p_1 = p_2$ does contribute. Using Lemma \ref{lem: kuznetsovapproximation}, we get that
\begin{eqnarray}
&&4\sum_{p_1 = p_2}\frac{1}{p_1^{1/2}p_2^{1/2}}\frac{\log p_1\log p_2}{\log^2 R}\widehat{\phi}_{1}\left(\frac{\log p_1}{\log R}\right)\widehat{\phi}_{2}\left(\frac{\log p_2}{\log R}\right)\sum_j\frac{h_T(t_j)}{\|u_j\|^2}\lambda_j(p_1)\overline{\lambda_j(p_2)}\nonumber\\
&&=4\sum_{p}\frac{1}{p}\frac{\log^2 p}{\log^2 R} \widehat{\phi}_{1}\left(\frac{\log p}{\log R}\right)\widehat{\phi}_{2}\left(\frac{\log p}{\log R}\right)
\left(\frac{LT}{\pi^2} + O\left(Le^{\pi/2L}\log T \cdot p^{1/2+\epsilon}\right)\right). \nonumber\\
&&=4\sum_{p}\frac{1}{p}\frac{\log^2 p}{\log^2 R} \widehat{\phi}_{1}\left(\frac{\log p}{\log R}\right)\widehat{\phi}_{2}\left(\frac{\log p}{\log R}\right)
\frac{LT}{\pi^2} + O\left(Le^{\pi/ 2L} T^{(1/2+\epsilon)(2\sigma)} \right). \nonumber\\
\end{eqnarray}
As $\frac{\pi}{2\log T} < (1-\eta)L$, $e^{\pi/2L} \ll T^{1-\eta}$, the error term does not contribute if
$1/2(2\sigma) < \eta$. Taking $\eta$ to be arbitrarily less than $1$, we see that the error term does not contribute if $\sigma < 1$.

When we divide by the sum of the weights (which is $LT/\pi^2 + o(1)$), we are left with a prime sum.
A standard computation (using partial summation and the Prime Number Theorem, see \cite{Mil0} for a proof) yields
\begin{eqnarray}
4\sum_{p}\frac{1}{p}\frac{\log^2 p}{\log^2 R}
\widehat{\phi}_{1}\left(\frac{\log p}{\log R}\right)
\widehat{\phi}_{2}\left(\frac{\log p}{\log R}\right)
\ = \  2 \int_{-\infty }^\infty  |z|\widehat{\phi}_1(z)\widehat{\phi}_2(z)dz+ O\left(\frac{\log \log R}{\log R}\right).\nonumber\\
\end{eqnarray}
Dividing by the weights, we get the desired result when $h_T = h_{2,T}$.

The case when $h_T = h_{1,T}$ is done in exactly the same way except we now use the first part of Lemma \ref{lem: kuznetsovapproximation} for $h_T = h_{1,T}$.
\end{proof}

\begin{lem}\label{lem: S1S2} If $R=T^2$ and $\frac{\pi}{2\log T} < (1-\eta)L = o(T^{1/8})$, the contribution from the $S_k(u_j,\phi_1)\overline{S_\ell(u_j,\phi_2)}$ terms are $O(\log\log R/\log R)$ if $(k,\ell) = (2,2)$ (in which case we may take $\sigma<\eta$ for $h_T = h_{2,T}$ and $\sigma < 1/4$ for $h_T = h_{1,T}$) or $(k,\ell) = (1,2)$ (in which case we may take $\sigma<\eta /2$ for $h_T = h_{2,T}$ and $\sigma< 1/8$ for
$h_T = h_{1,T}$).
\end{lem}

\begin{proof} The proof is similar to the previous lemma, following again by applications of the Kuznetsov trace formula. The support is slightly larger as the power of the primes in the denominator are larger.
\end{proof}

In the above lemmas, the worst restriction on the support is that $\sigma<\frac{1}{3}\eta$ (for $h_{2,T}$) and $\sigma<1/12$ (for $h_{1,T}$), coming from the $S_1(u_j,\phi_1)\overline{S_1(u_j,\phi_2)}$ term; it is typical in problems like this for these arguments to yield the 2-level with a support one-half that of the 1-level density.

\begin{proof}[Proof of Theorem \ref{thm: leveltwo}]
The proof follows immediately from substituting the above lemmas in the 2-level density expansion of Lemma \ref{eq:firstexpansion2levelfull}.
\end{proof}

\appendix


\section{Bounding errors in Bessel-Kloosterman terms}\label{sec:boundingerrorsBK}

We complete the proof of Lemma \ref{lem: kuznetsovapproximation} by showing how to handle the error terms. There are two sources of error, arising from truncating the original integral and from approximating the resulting expressions. The second is readily handled. Instead of approximating the terms in \eqref{eq:expansionsqrtu} and \eqref{eq:expansionuu}, we instead Taylor expand, and use the identity that the Fourier transform of $u^n h(u)$ at $\omega$ is proportional to $\widehat{h}^{(n)}(\omega)$; as $\widehat{h}$ is compactly supported, so too is its $n$\textsuperscript{th} derivative, and we again can use our compact support to restrict $c$. A standard error analysis shows that the errors from these series expansions are subsumed in the other errors, as \emph{now} we can exploit the additional $T$-decay. It is important that we only take finitely many terms in the Taylor expansions, as otherwise the big-Oh constants could depend on infinitely many Taylor coefficients. Fortunately it suffices to consider only finitely many terms, as a finite but large number can gain us any desired power of $T$, as $L/T^{1/8} = o(T^{-1/8})$ and $h$ is Schwartz (and so decays faster than any polynomial in the input).

We need a little more care in truncating the integral in \eqref{eq:ILTmnstart}. In \eqref{eq:ILTmnstart} we truncated the integral at $\pm \sqrt[4]{T}$. If $c < T^{\delta}$ for some \emph{fixed} $\delta > 0$, then the excised integration is negligible as the rapid decay of $h$ gives us any desired power savings of $c$ in the denominator, due to the fact that we are evaluating at more than $T^{1/4}/L \gg T^{1/8}$ units from $T$. It is essential that $\delta$ is fixed; if $c = e^T$, for example, then while the point of evaluation is far away on the scale of $T$, we would only have an arbitrary logarithmic savings with respect to $c$.

We are left with the contribution from large $c$. We may take $\delta = 2013$, so $c$ large means $c > T^{2013}$. Instead of bounding the error from truncating the integral in \eqref{eq:ILTmnstart}, we instead bound the integral in terms of $c$ and $T$. Using the standard property that $\Gamma(1-z)\Gamma(1+z) = \pi z /\sin(\pi z)$, and keeping only the first term in the expansion for $J_{2ir}$ from \eqref{eq:besselexpansionfromAS} (as before, the higher order terms give significantly less contributions which are subsumed in the zeroth order term), yields \bea\label{eq:errorintIc} I_{c, {\rm main}}(T,L;m,n) & \ = \ & \intii \left(\frac{4\pi\sqrt{mn}}{2c}\right)^{2ir} \frac{\sin(2\pi i r) \Gamma(1-2ir)}{2\pi i r} \frac{r h((r-T)/L)dr}{\cosh(\pi r)} \nonumber\\ &=& \frac1{\pi} \intii \left(\frac{2\pi\sqrt{mn}}{c}\right)^{2ir} \sinh(\pi r) \Gamma(1-2ir) h((r-T)/L)dr,\nonumber\\ \eea as \be \frac{\sin(2\pi i r)}{\cosh(\pi r)} \ = \ \frac{(e^{-2\pi r}-e^{2\pi r})/2i}{(e^{\pi r}+e^{-\pi r})/2} \ = \ i \frac{(e^{\pi r} - e^{-\pi r})(e^{\pi r}+e^{-\pi r})}{e^{\pi r}+e^{-\pi r}} \ = \ 2i \sinh(\pi r). \ee We now shift the contour in \eqref{eq:errorintIc}, moving $r$ to $r-\frac{k}{2}i$. If we take $0 < k < 1$, in particular $k = 3/4$, then we do not pass through any poles. Further, the factor $ \left(\frac{2\pi\sqrt{mn}}{c}\right)^{2ir}$ becomes $\left(\frac{2\pi\sqrt{mn}}{c}\right)^{2ir}  \left(\frac{2\pi\sqrt{mn}}{c}\right)^{3/4}$, which gives us plenty of additional $c$-decay. So long as the integrand is bounded, we will win, as $c \ge T^{2013}$. Shifting the contour, we find \bea I_{c, {\rm main}}(T,L;m,n) & \ = \ & \frac1{\pi} \left(\frac{2\pi\sqrt{mn}}{c}\right)^{3/4} \intii \left(\frac{2\pi\sqrt{mn}}{c}\right)^{2ir} \nonumber\\ & & \ \ \ \ \ \ \cdot \sinh\left(\pi r - \frac38\pi i\right) \Gamma\left(\frac14-2ir\right) h\left(\frac{r-T - \frac38i}{L}\right).\ \ \ \ \ \eea Straightforward algebra shows \be \sinh\left(\pi r - \frac38\pi i\right) \ = \ - i \sin\left(\frac{3\pi}8\right) \cdot \cosh(\pi r) + \cos\left(\frac{3\pi}8\right) \cdot \sinh(\pi r) \ee (and $\sinh(\pi r)$ and $\cosh(\pi r)$ both grow like $e^{\pi |r|}$ as $|r| \to \infty$). Also, by Fourier Inversion, the fact that the support of $\widehat{h}$ is contained in $(-1,1)$, and integrating by parts twice we find \bea h\left(\frac{r-T - \frac38i}{L}\right) & \ = \ & \int_{-1}^1 \widehat{h}(y) e^{2 \pi i \left( \frac{r-T}{L} - \frac{3i}{8L}\right) y} dy \nonumber\\  & \ = \ & \left( \frac{r-T}{L} - \frac{3i}{8L}\right)^{-2} \int_{-1}^1 \widehat{h}''(y) e^{2 \pi i \left( \frac{r-T}{L} - \frac{3i}{8L}\right) y} dy \nonumber\\ &=& \left( \frac{r-T}{L} - \frac{3i}{8L}\right)^{-2} \int_{-1}^1 \widehat{h}''(y) e^{2 \pi i \left( \frac{r-T}{L}\right)y} e^{3\pi y/4L} dy \nonumber\\ & \ll_h & \frac{L^2}{(r-T)^2 + 9/64}  e^{3\pi/4L} \nonumber\\ & \ll &  \ \frac{L^2T^{\frac32(1-\eta)}}{(r-T)^2+9/64}, \eea where the last inequality follows from the assumption that $\frac{\pi}{2\log T} < (1-\eta)L$. We substitute our bounds into $I_{c, {\rm main}}(L,T;m,n)$, and note that standard properties of the Gamma and hyperbolic functions give $\left|\cosh(\pi r) \Gamma(1/4 - 2 i r)\right| \ll r^{-1/4}$. To see this, by Stirling's formula ($\Gamma(z) = z^{z-1/2} e^{-z} \sqrt{2\pi}(1 + O(1/z))$) and the series expansion for arctangent ($\arctan \phi = \phi + O(\phi)$ for $\phi$ small) we have for $r > 0$ that \bea \Gamma(1/4 - 2ir) & \ = \ & (1/4 - 2ir)^{1/4-2ir-1/2} e^{-(1/4-2ir)} \sqrt{2\pi} \left(1 + O(1/r)\right) \nonumber\\ & = & \left( \sqrt{4r^2+1/16}\ e^{-i(\pi/2 - \arctan(1/8r))} \right)^{-1/4 + 2ir} e^{-(1/4-2ir)} \sqrt{2\pi} \left(1 + O(1/r)\right) \nonumber\\ & \ll & r^{-1/4} \left(e^{-i(\pi/2 - 1/8r + O(1/r^3))}\right)^{-2ir} \nonumber\\ & \ll & r^{-1/4} e^{-\pi r -1/4 + O(1/r^2)} \ \ll \ r^{-1/4} e^{-\pi r}; \eea as $\sinh(\pi r), \cosh(\pi r) \ll e^{\pi r}$, our claimed bound follows. The proof for $r < 0$ follows similarly. We find \bea I_{c, {\rm main}}(L,T;m,n) & \ \ll \ & \frac{(mn)^{3/8}}{c^{3/4}} \intii \frac{L^2T^{\frac32(1-\eta)}}{(r-T)^2+9/64} dr \nonumber\\ & \ \ll \ &  \frac{(mn)^{3/8} L^2 T^{\frac32(1-\eta)}}{c^{3/4}}. \eea Remembering that we only need this estimate for $c \ge T^{2013}$, we see the contribution these Bessel integrals in the sums weighted by Kloosterman factors are negligible, as we will have a $c^{5/4-\epsilon}$ in the denominator. \hfill $\Box$


\ \\


\begin{thebibliography}{DHKM} 




\bibitem[AS]{AS}
M. Abramowitz and I. A. Stegun, \emph{Handbook of Mathematical Functions with Formulas, Graphs, and Mathematical Tables}, 9th printing, New York: Dover, 1972.

\bibitem[Con]{Con}
J. B. Conrey, \emph{$L$-Functions and random matrices}. Pages
331--352 in \emph{Mathematics unlimited --- 2001 and Beyond},
Springer-Verlag, Berlin, 2001.

%
\bibitem[CI]{CI}
\newblock J. B. Conrey and H. Iwaniec, \emph{Spacing of Zeros of Hecke L-Functions and the Class Number Problem}, Acta Arith. \textbf{103} (2002) no. 3, 259--312.



\bibitem[Da]{Da}
\newblock H. Davenport, \emph{Multiplicative Number Theory, $2$nd edition},
Graduate Texts in Mathematics \textbf{74}, Springer-Verlag, New York,
$1980$, revised by H. Montgomery.

\bibitem[DI]{DI}
\newblock J. M. Deshouillers and H. Iwaniec, \emph{Kloosterman sums and Fourier coefficients of cusp
forms}, Invent. Math. \textbf{70} (1982/83), 219--288.

\bibitem[DM1]{DM1}
\newblock E. Due\~nez and S. J. Miller, \emph{The low lying zeros of a
$\text{GL}(4)$ and a $\text{GL}(6)$ family of $L$-functions},
Compositio Mathematica \textbf{142} (2006), no. 6, 1403--1425.

\bibitem[DM2]{DM2}
\newblock E. Due\~nez and S. J. Miller, \emph{The effect of
convolving families of $L$-functions on the underlying group
symmetries},Proceedings of the London Mathematical Society, 2009; doi: 10.1112/plms/pdp018.

\bibitem[FM]{FM}
F. W. K. Firk and S. J. Miller, \emph{Nuclei, Primes and the Random Matrix Connection}, Symmetry \textbf{1} (2009), 64--105; doi:10.3390/sym1010064.


\bibitem[FI]{FI}
\newblock E. Fouvry and H. Iwaniec, \emph{Low-lying zeros of dihedral
$L$-functions}, Duke Math. J.  \textbf{116} (2003),  no. 2, 189-217.

\bibitem[Gao]{Gao}
\newblock P. Gao, \emph{$N$-level density of the low-lying zeros of
quadratic Dirichlet $L$-functions}, Ph.~D thesis, University of
Michigan, 2005.

\bibitem[Go]{Go}
\newblock D. Goldfeld, \emph{The class number of quadratic fields and the conjectures of Birch and Swinnerton-Dyer}, Ann. Scuola Norm. Sup. Pisa (4) \textbf{3} (1976), 623--663.

%
\bibitem[GZ]{GZ}
\newblock B. Gross and D. Zagier, \emph{Heegner points and derivatives of $L$-series}, Invent. Math \textbf{84} (1986), 225--320.
%

\bibitem[G\"u]{Gu}
\newblock A. G\"ulo\u{g}lu, \emph{Low-Lying Zeros of Symmetric
Power $L$-Functions}, Internat. Math. Res. Notices 2005, no. 9,
517-550.

\bibitem[Ha]{Ha}
B. Hayes, \emph{The spectrum of Riemannium}, American Scientist
\textbf{91} (2003), no. 4, 296--300.

\bibitem[Hej]{Hej}
\newblock D. Hejhal, \emph{On the triple correlation of zeros of
the zeta function}, Internat. Math. Res. Notices 1994, no. 7,
294-302.

\bibitem[HL]{HL}
\newblock J. Hoffstein and P. Lockhart, \emph{Coefficients of Maass Forms and the Siegel Zero}, The Annals of Mathematics \textbf{140} (1994), second series, no. 1, 161--176.



\bibitem[HM]{HM}
\newblock C. Hughes and S. J. Miller, \emph{Low-lying zeros of $L$-functions
with orthogonal symmtry}, Duke Math. J., \textbf{136}
(2007), no. 1, 115--172.

\bibitem[HR]{HR}
\newblock C. Hughes and Z. Rudnick, \emph{Linear Statistics of
Low-Lying Zeros of $L$-functions},  Quart. J. Math. Oxford
\textbf{54} (2003), 309--333.


\bibitem[Iw1]{Iw1}
\newblock H. Iwaniec, \emph{Small eigenvalues of Laplacian for $\Gamma_0(N)$}, Acta Arith. \textbf{16} (1990), 65--82.

\bibitem[Iw2]{Iw2}
H. Iwaniec, \emph{Introduction to the Spectral Theory of Automorphic Forms}, Biblioteca de la Revista Matem\'atica Iberoamericana, 1995.

\bibitem[IK]{IK}
H. Iwaniec and E. Kowalski, \emph{Analytic Number Theory}, AMS
Colloquium Publications, Vol.
textbf{53}, AMS, Providence, RI, 2004.


\bibitem[ILS]{ILS}
\newblock H. Iwaniec, W. Luo and P. Sarnak, \emph{Low lying zeros of
families of $L$-functions}, Inst. Hautes �tudes Sci. Publ. Math.
\textbf{91}, 2000, 55--131.

\bibitem[KaSa1]{KaSa1}
\newblock N. Katz and P. Sarnak, \emph{Random Matrices, Frobenius
Eigenvalues and Monodromy}, AMS Colloquium Publications \textbf{45},
AMS, Providence, $1999$.

\bibitem[KaSa2]{KaSa2}
\newblock N. Katz and P. Sarnak, \emph{Zeros of zeta functions and symmetries},
Bull. AMS \textbf{36}, $1999$, $1-26$.

%
%

\bibitem[K]{K}
\newblock H. Kim, \emph{Functoriality for the exterior square of $GL_2$ and the symmetric fourth of $GL_2$}, Jour. AMS  \textbf{16} (2003), no. 1, 139--183.

\bibitem[KSa]{KSa}
\newblock H. Kim and P. Sarnak, \emph{Appendix: Refined estimates towards the Ramanujan
and Selberg conjectures}, Appendix to \cite{K}.





\bibitem[Liu]{Liu}
\newblock J. Liu, \emph{Lectures On Maass Forms}, Postech, March 25--27, 2007. \texttt{http://www.prime.sdu.edu.cn/lectures/LiuMaassforms.pdf}.

\bibitem[LiuYe1]{LY1}
\newblock J. Liu and Y. Ye, \emph{Subconvexity for Rankin-Selberg $L$-Functions for Maass Forms},
Geom. funct. anal. \textbf{12} (2002), 1296--1323.

\bibitem[LiuYe2]{LY2}
\newblock J. Liu and Y. Ye, \emph{Petersson and Kuznetsov trace formulas}, in Lie groups and automorphic forms (Lizhen Ji, Jian-Shu Li, H. W. Xu and Shing-Tung Yau editors), AMS/IP Stud. Adv. Math. \textbf{37}, AMS, Providence, RI, 2006, pages 147--168.

\bibitem[Mil1]{Mil0}
\newblock S. J. Miller, \emph{$1$- and $2$-level densities for families of elliptic
curves: evidence for the underlying group symmetries}, Princeton University, Ph.\ D. thesis, 2002.

\bibitem[Mil2]{Mil1}
\newblock S. J. Miller, \emph{$1$- and $2$-level densities for families of elliptic
curves: evidence for the underlying group symmetries}, Compositio
Mathematica \textbf{140} (2004), 952--992.







\bibitem[MilMo]{MilMo}
\newblock S. J. Miller and D. Montague, \emph{An Orthogonal Test of the $L$-functions Ratios Conjecture, II}, Acta Arith. \textbf{146} (2011), 53--90.

\bibitem[MT-B]{MT-B}
\newblock S. J. Miller and R. Takloo-Bighash, \emph{An Invitation to Modern Number Theory}, Princeton University Press, Princeton, NJ, 2006.

\bibitem[MilPe]{MilPe}
\newblock S. J. Miller and R. Peckner, \emph{Low-lying zeros of number field $L$-functions}, preprint.

\bibitem[Mon]{Mon}
\newblock H. Montgomery, \emph{The pair correlation of zeros of the zeta
function}, Analytic Number Theory, Proc. Sympos. Pure Math.
\textbf{24}, Amer. Math. Soc., Providence, 1973, 181-193.



\bibitem[Od1]{Od1}
\newblock A. Odlyzko, \emph{On the distribution of spacings
between zeros of the zeta function}, Math. Comp. \textbf{48} (1987),
no. 177, 273--308.

\bibitem[Od2]{Od2}
\newblock A. Odlyzko, \emph{The $10^{22}$-nd zero of the Riemann zeta function}, Proc.
Conference on Dynamical, Spectral and Arithmetic Zeta-Functions, M.
van Frankenhuysen and M. L. Lapidus, eds., Amer. Math. Soc.,
Contemporary Math. series, 2001,
\texttt{http://www.research.att.com/$\sim$amo/doc/zeta.html}.


\bibitem[OS]{OS}
\newblock A. E. \"Ozl\"uk and C. Snyder, \emph{Small zeros of quadratic $L$-functions},
Bull. Austral. Math. Soc. \textbf{47} (1993), no. 2, 307--319.

%


\bibitem[RR]{RR}
\newblock G. Ricotta and E. Royer, \emph{Statistics for low-lying
zeros of symmetric power $L$-functions in the level aspect},
preprint, to appear in Forum Mathematicum.

\bibitem[Ro]{Ro}
\newblock E. Royer, \emph{Petits z\'{e}ros de fonctions $L$
de formes modulaires}, Acta Arith. \textbf{99} (2001),  no. 2,
147-172.

\bibitem[Rub]{Rub}
\newblock M. Rubinstein, \emph{Low-lying zeros of $L$--functions
and random matrix theory}, Duke Math. J. \textbf{109}, (2001),
147--181.

\bibitem[RS]{RS}
\newblock Z. Rudnick and P. Sarnak, \emph{Zeros of principal $L$-functions
 and random matrix theory}, Duke Math. J. \textbf{81},
$1996$, $269-322$.

\bibitem[Sar]{Sar}
\newblock P. Sarnak, \emph{Letter to Zeev Rudnick}, October 2002.





\bibitem[Ya]{Ya}
A. Yang, \emph{Low-lying zeros of Dedekind zeta functions attached to cubic number fields}, preprint.


\bibitem[Yo]{Yo2}
\newblock M. Young, \emph{Low-lying zeros of families of elliptic curves},
J. Amer. Math. Soc. \textbf{19} (2006), no. 1, 205--250.



\end{thebibliography}
\end{document}